\documentclass[a4paper,11pt]{amsart}
\usepackage{mathptmx}
\usepackage{amsmath}
\usepackage{amscd}
\usepackage{amssymb}
\usepackage{amsthm}
\usepackage{xspace}
\usepackage[all,tips]{xy}
\usepackage[dvips]{graphicx}
\usepackage{verbatim}
\usepackage{syntonly}
\usepackage{hyperref}
\usepackage{float}
\usepackage{amsmath, amsthm, graphics, amssymb,fullpage,color, epsfig,url, wrapfig}
\usepackage{indentfirst}
\usepackage{esint, ulem}
\usepackage{tikz, multicol}
\usetikzlibrary{calc}

\providecommand{\MR}{\relax\ifhmode\unskip\space\fi MR }

\providecommand{\href}[2]{#2}

\theoremstyle{plain}
\newtheorem{thm}{Theorem}
\newtheorem{lem}[thm]{Lemma}
\newtheorem{prop}[thm]{Proposition}
\newtheorem{defn}[thm]{Definition}
\newtheorem{cor}[thm]{Corollary}
\newtheorem*{example}{Example}

\theoremstyle{definition}
\newtheorem{rem}{Remark}

\newtheorem*{thms}{Theorem}
\newtheorem*{clm}{Claim}

\newcommand{\disp}{\displaystyle}

\DeclareMathOperator{\dist}{dist}

\DeclareMathOperator{\proj}{proj}
\DeclareMathOperator{\Fav}{Fav}
\DeclareMathOperator{\FavC}{\Fav_{\mathcal{C}}}

\newcommand{\eps}{\varepsilon}
\newcommand{\vp}{\varphi}

\newcommand{\al}{\alpha}
\newcommand{\be}{\beta}
\newcommand{\ga}{\gamma}
\newcommand{\de}{\delta}

\newcommand{\te}{\theta}
\newcommand{\la}{\lambda}

\newcommand{\om}{\omega}

\newcommand{\nid}{\noindent}

\newcommand{\iny}{\infty}

\newcommand{\su}{\subset}

\newcommand{\abs}[1]{\left\vert#1\right\vert}
\newcommand{\set}[1]{\left\{#1\right\}}
\newcommand{\brac}[1]{\left[#1\right]}
\newcommand{\pr}[1]{\left( #1 \right) }
\newcommand{\pb}[1]{\left( #1 \right] }
\newcommand{\bp}[1]{\left[ #1 \right) }

\newcommand{\N}{\ensuremath{\mathbb{N}}}

\newcommand{\R}{\ensuremath{\mathbb{R}}}
\newcommand{\Z}{\ensuremath{\mathbb{Z}}}

\numberwithin{equation}{section}

\title{Upper and lower bounds on the rate of decay of the \\ Favard curve length for the four-corner Cantor set}
\author[Cladek, Davey, Taylor]{Laura Cladek \and Blair Davey \and Krystal Taylor}
\address{Laura Cladek, Department of Mathematics, University of California, Los Angeles}
\email{cladek@math.ucla.edu}
\thanks{Cladek is supported in part by the National Science Foundation Postdoctoral Fellowship 1703715.}
\address{Blair Davey, Department of Mathematics, City College of New York, CUNY}
\email{bdavey@ccny.cuny.edu}
\thanks{Davey is supported in part by the Simons Foundation Grant 430198.}
\address{Krystal Taylor, Department of Mathematics, The Ohio State University}
\email{taylor.2952@osu.edu}
\thanks{Taylor is supported in part by the Simons Foundation Grant 523555.}

\date{}

\begin{document}

\begin{abstract}
The Favard length of a subset of the plane is defined as the average of its orthogonal projections.
This quantity is related to the probabilistic Buffon needle problem; that is, the Favard length of a set is proportional to the probability that a needle or a line that is dropped at random onto the set will intersect the set.
If instead of dropping lines onto a set, we drop fixed curves, then the associated Buffon curve probability is proportional to the so-called Favard curve length.
In this article, we estimate upper and lower bounds for the rate of decay of the Favard curve length of the four-corner Cantor set.
Our techniques build on the ideas that have been previously used for the classical Favard length.
\end{abstract}

\maketitle

\section{Introduction}

Let $E$ be a subset of the unit square $\brac{0, 1}^2$.
The {\bf Buffon needle problem} asks the likelihood that a needle, or a line, that is dropped at random onto the plane intersects the set $E$ given that it intersects $\brac{0,1}^2$.
More rigorously, we are seeking the probability that $\ell \cap E \ne \emptyset$, where $\ell$ is a line with independent, uniformly distributed orientation and distance from the origin after conditioning to the event that the line intersects $\brac{0,1}^2$.
This quantity is given by
$$\mathbf{P} := P\pr{\ell \cap E \ne \emptyset : \ell \textrm{ is any line in $\R^2$ for which } \ell \cap \brac{0,1}^2 \ne \emptyset}.$$
If we parametrize all such lines by letting $\ell_{\be, \om}$ denote the line passing through $\pr{0, \be}$ with direction orthogonal to $\om \in \mathbb{S}^1$, then
\begin{align*}
\mathbf{P} 
&\simeq \abs{\set{\pr{\be, \om} \in \R \times \mathbb{S}^1 : E \cap \ell_{\be, \om} \ne \emptyset}},
\end{align*}
where $\abs{\cdot}$ is used to denote the Lebesgue measure and $A \simeq B$ means that both $A \lesssim B$ and $B \lesssim A$ hold, where $A \lesssim B$ means that $A \le c B$ for a constant $c > 0$.
Observe that for a fixed $\om \in \mathbb{S}^1$, 
\begin{align*}
\set{\be \in \R : E \cap \ell_{\be, \om} \ne \emptyset} = \proj_{\om}\pr{E},
\end{align*}
where $\proj_{\om}\pr{S}$ denotes the linear projection of a set $S$ onto the angle $\om$.
An application of Fubini's theorem shows that
\begin{align}
\mathbf{P}
&\simeq \int_{\mathbb{S}^1} \abs{\set{\be \in \R : E \cap \ell_{\be, \om} \ne \emptyset}} d\om
= \int_{\mathbb{S}^1} \abs{\proj_{\om}\pr{E}} d\om
=: \Fav\pr{E}.
\end{align}
Therefore, the {\bf Favard length} is connected to the classical Buffon needle problem.

Now we ask what happens when lines are replaced by more general curves.
Let $\mathcal{C}$ denote a curve in $\R^2$.
We seek the probability that $\mathcal{C}$ intersects $E$ when $\mathcal{C}$ is dropped randomly onto the plane so that it intersects $\brac{0,1}^2$.
When the curve $\mathcal{C}$ is dropped, we allow for it to be translated but not rotated.
Assuming that $\mathcal{C}$ is of finite length, this probability satisfies
\begin{align*}
\mathbf{P}_{\mathcal{C}} 
&\simeq \abs{\set{\pr{\al, \be} \in \R^2 : E \cap \pr{\pr{\al, \be} + \mathcal{C}} \ne \emptyset}},
\end{align*}
where $\pr{\al, \be} + \mathcal{C} = \set{\pr{\al, \be} + z: z \in \mathcal{C}}$.
Observe that $E \cap \pr{\pr{\al, \be} + \mathcal{C}} \ne \emptyset$ iff $\pr{\al, \be} \in E - \mathcal{C}$, where $E - \mathcal{C} = \set{e - z : e \in E, z \in \mathcal{C}}$.
To draw a parallel between this problem and the classical Buffon needle problem, we introduce a family of curve projections.
Given $\al \in \R$ and $p \in \R^2$, let $\Phi_\al\pr{p}$ denote the set of $y$-coordinates of the intersection of  $p - \mathcal{C}$ with the line $x = \al$.
That is, 
\begin{equation}
\label{PaProj}
\Phi_\al\pr{p} = \set{\be \in \R : \pr{\al, \be} \in \pr{p - \mathcal{C}} \cap \set {x = \al}}.
\end{equation}
The map $\Phi_\al\pr{p}$ can be viewed as an analog of $\proj_\om$.
Given $\be \in \R$, the inverse set $\Phi_\al^{-1}\pr{\be} = \set{p : \be \in \Phi_\al\pr{p}}$ is given by $\pr{\al, \be} + \mathcal{C}$.
With this new notation, we see that
\begin{align*}
\mathbf{P}_{\mathcal{C}}
&\simeq \abs{\set{\pr{\al, \be} \in \R^2 : E \cap \Phi_\al^{-1}\pr{\be} \ne \emptyset}}.
\end{align*}
And for each fixed $\al \in \R$, we have
\begin{align*}
\set{\be \in \R :  E \cap \Phi_\al^{-1}\pr{\be} \ne \emptyset}
&= \Phi_\al\pr{E}.
\end{align*}
As above, an application of Fubini's theorem shows that
\begin{align*}
\mathbf{P}_{\mathcal{C}}
&\simeq \int_\R \abs{\set{\be :  E \cap \Phi_\al^{-1}\pr{\be} \ne \emptyset}} d\al
= \int_\R \abs{\Phi_\al\pr{E}} d\al
= \FavC \pr{E}.
\end{align*}
This shows that the {\bf Favard curve length} is comparable to the probability associated to the {\bf Buffon curve problem.}

Now we give the formal definition of the Favard curve length.

\begin{defn}[Favard curve length]
\label{FavC}
Let $\mathcal{C}$ be a curve in $\R^2$ with a family of curve projections defined as in \eqref{PaProj}.
If $E \su \R^2$, then the {\em Favard curve length of $E$} is given by 
\begin{equation*}
\FavC\pr{E} 
:= \abs{\set{\pr{\al, \be} \in \R^2 : \Phi_\al^{-1}\pr{\be} \cap E \ne \emptyset}}
= \int_{\R} \abs{\Phi_\al\pr{E}} d\al.
\end{equation*}
\end{defn}

\begin{rem}
\label{FavCLengthRem}
Observe that 
\begin{align*}
\FavC\pr{E} 
&= \abs{\set{\pr{\al, \be} \in \R^2 : \pr{\pr{\al, \be} + \mathcal{C}} \cap E \ne \emptyset}} \\
&= \abs{\set{\pr{\al, \be} \in \R^2 : \pr{\al, \be} \cap \pr{E +\pr{-\mathcal{C}}} \ne \emptyset}} 
= \abs{E + ( - \mathcal{C})}.
\end{align*}
Therefore, the Favard curve length can be interpreted as the measure of a sum set.
Moreover, although we introduced $\Phi_\al$ (the set of $y$-values of the intersection of $p - \mathcal{C}$ with a vertical line defined by $x = \al$) to define the Favard curve length, a version of Definition \ref{FavC} still holds for any other choice of orthogonal basis.
For example, we could define $\Psi_\be$ to be the set of $x$-values of the intersection of $p - \mathcal{C}$ with a horizontal line $y = \be$,
$$\Psi_\be\pr{p} = \set{\al \in \R : \pr{\al, \be} \in \pr{p - \mathcal{C}} \cap \set {y = \be}}.$$
Following the arguments from above, we would compute the Favard curve length by integrating over $\be$
to get that
$$\FavC\pr{E} 
:= \abs{\set{\pr{\al, \be} \in \R^2 : \Psi_\be^{-1}\pr{\be} \cap E \ne \emptyset}}
= \int_{\R} \abs{\Psi_\be\pr{E}} d\be.$$
\end{rem}

The Favard length is of interest because of its connection to the Buffon needle problem, but it also gives important information about the rectifiability of the set.
The Besicovitch projection theorem \cite{Bes39}, \cite{Fal80}, \cite[Theorem 6.13]{Fal85}, \cite[Theorem 18.1]{Mat95} states that if a subset $E$ of the plane has finite length in the sense of Hausdorff measure and is purely unrectifiable (so that its intersection with any Lipschitz graph has zero length), then almost every linear projection of $E$ to a line will have zero measure. 
This means that if $E$ is purely unrectifiable, then the Favard length of $E$ is zero.
In our companion paper \cite{DT21}, we study quantitative versions of this statement for the Favard curve length.
The results of \cite{DT21} (see also \cite{ST17} and \cite{HJJL} where nonlinear versions of the Besicovitch projection theorem were originally attained) imply the following:

\begin{thms}[Besicovitch generalized projection theorem]
Let $E \su \R^2$ be such that $\mathcal{H}^1\pr{E} \in \pr{0, \iny}$.
Assume that $\mathcal{C}$ is piecewise $C^1$ with a piecewise bi-Lipschitz continuous unit tangent vector.
If $E$ is purely unrectifiable, then $\FavC\pr{E} = 0$.
\end{thms}

Our paper \cite{DT21} follows the viewpoint of Tao \cite{Tao09} and uses multi-scale analysis to quantify the previous statement.
Roughly speaking, we show that if $E$ is close to being purely unrectifiable, then for an appropriate class of curves, the Favard curve length of $E$ will be very small.
In the current article, we seek to quantify this statement through a different viewpoint.
Let $\disp K = \bigcap_{n=1}^\iny K_n$ denote the four-corner Cantor set (which we rigorously introduce below).
Since $K$ is a compact, purely unrectifiable set with bounded, non-zero $\mathcal{H}^1$-measure, then $\FavC\pr{K} = 0$.
In particular, it follows from the dominated convergence theorem that $\disp \lim_{n \to \iny} \FavC\pr{K_n} = 0$.
Our current approach to quantifying the theorem stated above is to find upper and lower bounds for $\FavC\pr{K_n}$ as a function of $n$.

In recent years, there has been significant interest in determining rates of decay of the classical Favard length for fractal sets.
In \cite{PS02}, Peres and Solomyak proved that $\Fav\pr{K_n} \lesssim \exp\pr{- c \log_* n}$, where $\log_* y = \min \set{ m \ge 0 : \log^{m}\pr{y} \le 1}$ denotes the inverse tower function.
They also investigated Favard length bounds for other self-similar sets and random Cantor sets.
The upper bound of Peres and Solomyak was greatly improved by Nazarov, Peres, and Volberg in \cite{NPV10}, where they proved the following:  

\begin{thm}[Nazarov, Peres, Volberg]
\label{NPVThm}
For every $p > 6$, there exists $C>0$ such that for all $n \in \N$,
$$\Fav(K_n)\le C n^{-1/p}.$$
\end{thm}

\nid
In \cite{LZ10}, \L aba and Zhai considered more general product Cantor sets of the form $\disp E = \bigcap E_n$ and showed that there exists $p \in \pr{6, \iny}$ so that $\Fav\pr{E_n} \lesssim n^{-1/p}$.
In a related direction, Bond and Volberg showed in \cite{BoV10} that $\Fav\pr{\mathcal{G}_n} \lesssim n^{-1/14}$, where $\mathcal{G}_n$ is a $3^{-n}$-approximation to the Sierpinski gasket.
With $\disp S = \bigcap S_n$ denoting a more general self-similar set, Bond and Volberg showed that $\Fav\pr{S_n} \le \exp\pr{- c \sqrt{ \log n}}$ in \cite{BoV12}.
All of these results were generalized by Bond, \L aba and Volberg in \cite{BLV14} where they considered self-similar rational product Cantor sets.
Under certain assumptions on $S$, it was shown that  $\Fav\pr{S_n} \le n^{-p/\log\log n}$, improving on the results of \cite{BoV12}.
For the four-corner Cantor set, the following lower bound was established by Bateman and Volberg in \cite{BaV10}. 

\begin{thm}[Bateman, Volberg]
\label{BVThm}
There exists $C>0$ such that for all $n \in \N$,
$$\Fav(K_n)\ge C n^{-1} \log n.$$
\end{thm}

\nid
The upper and lower bounds described by Theorems \ref{NPVThm} and \ref{BVThm}, respectively, are currently the best known results.

A variation of the classical Favard length was considered by Bond and Volberg in \cite{BoV11}.
Their so-called ``circular Favard length" replaces linear projections of $K_n$ with circular projections of $K_n$, where the radius of each circle depends on $n$.
We point out that our approach is different from theirs since our curves do not vary with $n$.
Since the Favard curve length may be interpreted as the measure of the sum set $E + \pr{- \mathcal{C}} = E - \mathcal{C}$ (see Remark \ref{FavCLengthRem}), our work is closely related to the ideas in \cite{ST17} and \cite{ST20} where the dimension, measure, and interior of such sum sets were studied.
Simon and Taylor showed that the 2-dimensional Lebesgue measure of $E + \mathbb{S}^1$ equals zero iff $E$ is an irregular $1$-set.
In our language, this means that $\Fav_{\mathbb{S}^1}\pr{E} = 0$ iff $E$ is a purely unrectifiable $1$-set with finite, non-zero measure.

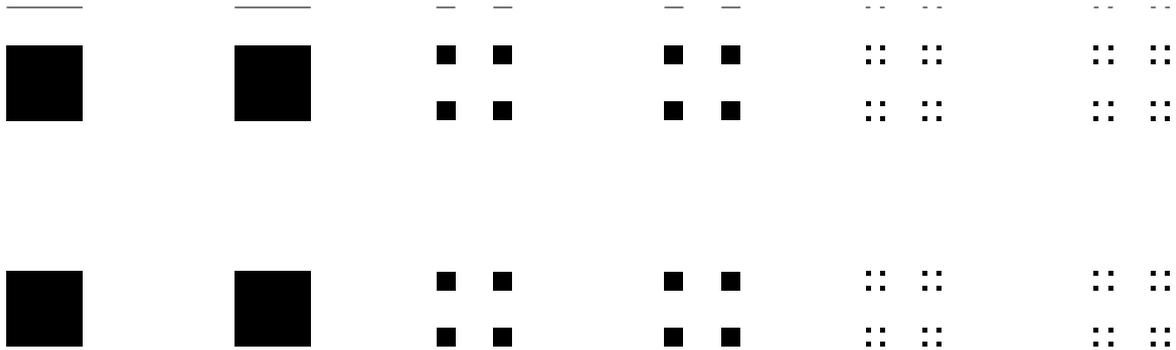
\begin{figure}[h]
\begin{tikzpicture}
\fill[black]
 (0, 0) -- (0, 1) --
(0, 1) -- (1, 1) --
(1, 1) -- (1, 0) -- 
(1, 0) -- (0, 0) -- 
   cycle;
\fill[black]
 (3, 0) -- (3, 1) --
(3, 1) -- (3, 1) --
(4, 1) -- (4, 0) -- 
(4, 0) -- (3, 0) -- 
   cycle;
\fill[black]
 (0, 3) -- (0, 4) --
(0, 4) -- (1, 4) --
(1, 4) -- (1, 3) -- 
(1, 3) -- (0, 3) -- 
   cycle;
\fill[black]
 (3, 3) -- (3, 4) --
(3, 4) -- (4, 4) --
(4, 4) -- (4, 3) -- 
(4, 3) -- (3, 3) -- 
   cycle;
\draw (0, 4.5) -- (1, 4.5);
\draw (3, 4.5) -- (4, 4.5);
\end{tikzpicture}
\qquad\qquad
\begin{tikzpicture}
\fill[black]
 (0, 0) -- (0, .25) --
(0, .25) -- (.25, .25) --
(.25, .25) -- (.25, 0) -- 
(.25, 0) -- (0, 0) -- 
   cycle;
\fill[black]
 (.75, 0) -- (.75, .25) --
(.75, .25) -- (.75, .25) --
(1, .25) -- (1, 0) -- 
(1, 0) -- (.75, 0) -- 
   cycle;
\fill[black]
 (0, .75) -- (0, 1) --
(0, 1) -- (.25, 1) --
(.25, 1) -- (.25, .75) -- 
(.25, .75) -- (0, .75) -- 
   cycle;
\fill[black]
 (.75, .75) -- (.75, 1) --
(.75, 1) -- (1, 1) --
(1, 1) -- (1, .75) -- 
(1, .75) -- (.75, .75) -- 
   cycle;
\fill[black]
 (3, 0) -- (3, .25) --
(3, .25) -- (3.25, .25) --
(3.25, .25) -- (3.25, 0) -- 
(3.25, 0) -- (3, 0) -- 
   cycle;
\fill[black]
 (3.75, 0) -- (3.75, .25) --
(3.75, .25) -- (3.75, .25) --
(4, .25) -- (4, 0) -- 
(4, 0) -- (3.75, 0) -- 
   cycle;
\fill[black]
 (3, .75) -- (3, 1) --
(3, 1) -- (3.25, 1) --
(3.25, 1) -- (3.25, .75) -- 
(3.25, .75) -- (3, .75) -- 
   cycle;
\fill[black]
 (3.75, .75) -- (3.75, 1) --
(3.75, 1) -- (4, 1) --
(4, 1) -- (4, .75) -- 
(4, .75) -- (3.75, .75) -- 
   cycle;
\fill[black]
 (0, 3) -- (0, 3.25) --
(0, 3.25) -- (.25, 3.25) --
(.25, 3.25) -- (.25, 3) -- 
(.25, 3) -- (0, 3) -- 
   cycle;
\fill[black]
 (.75, 3) -- (.75, 3.25) --
(.75, 3.25) -- (.75, 3.25) --
(1, 3.25) -- (1, 3) -- 
(1, 3) -- (.75, 3) -- 
   cycle;
\fill[black]
 (0, 3.75) -- (0, 4) --
(0, 4) -- (.25, 4) --
(.25, 4) -- (.25, 3.75) -- 
(.25, 3.75) -- (0, 3.75) -- 
   cycle;
\fill[black]
 (.75, 3.75) -- (.75, 4) --
(.75, 4) -- (1, 4) --
(1, 4) -- (1, 3.75) -- 
(1, 3.75) -- (.75, 3.75) -- 
   cycle;
\fill[black]
 (3, 3) -- (3, 3.25) --
(3, 3.25) -- (3.25, 3.25) --
(3.25, 3.25) -- (3.25, 3) -- 
(3.25, 3) -- (3, 3) -- 
   cycle;
\fill[black]
 (3.75, 3) -- (3.75, 3.25) --
(3.75, 3.25) -- (3.75, 3.25) --
(4, 3.25) -- (4, 3) -- 
(4, 3) -- (3.75, 3) -- 
   cycle;
\fill[black]
 (3, 3.75) -- (3, 4) --
(3, 4) -- (3.25, 4) --
(3.25, 4) -- (3.25, 3.75) -- 
(3.25, 3.75) -- (3, 3.75) -- 
   cycle;
\fill[black]
 (3.75, 3.75) -- (3.75, 4) --
(3.75, 4) -- (4, 4) --
(4, 4) -- (4, 3.75) -- 
(4, 3.75) -- (3.75, 3.75) -- 
   cycle;
\draw (0, 4.5) -- (0.25, 4.5);
\draw (0.75, 4.5) -- (1, 4.5);
\draw (3, 4.5) -- (3.25, 4.5);
\draw (3.75, 4.5) -- (4, 4.5);
\end{tikzpicture}
\qquad\qquad
\begin{tikzpicture}
\fill[black]
 (0, 0) -- (0, .0625) --
(0, .0625) -- (.0625, .0625) --
(.0625, .0625) -- (.0625, 0) -- 
(.0625, 0) -- (0, 0) -- 
   cycle;
\fill[black]
 (0.1875, 0) -- (0.1875, .0625) --
(0.1875, .0625) -- (.25, .0625) --
(.25, .0625) -- (.25, 0) -- 
(.25, 0) -- (0.1875, 0) -- 
   cycle;
\fill[black]
 (0, 0.1875) -- (0, .25) --
(0, .25) -- (.0625, .25) --
(.0625, .25) -- (.0625, 0.1875) -- 
(.0625, 0.1875) -- (0, 0.1875) -- 
   cycle;
\fill[black]
 (0.1875, 0.1875) -- (0.1875, .25) --
(0.1875, .25) -- (.25, .25) --
(.25, .25) -- (.25, 0.1875) -- 
(.25, 0.1875) -- (0.1875, 0.1875) -- 
   cycle;
\fill[black]
 (.75, 0) -- (.75, .0625) --
(.75, .0625) -- (.75, .0625) --
(0.8125, .0625) -- (0.8125, 0) -- 
(0.8125, 0) -- (.75, 0) -- 
   cycle;
\fill[black]
 (0.9375, 0) -- (0.9375, .0625) --
(0.9375, .0625) -- (0.9375, .0625) --
(1, .0625) -- (1, 0) -- 
(1, 0) -- (0.9375, 0) -- 
   cycle;
\fill[black]
 (.75, 0.1875) -- (.75, .25) --
(.75, .25) -- (.75, .25) --
(0.8125, .25) -- (0.8125, 0.1875) -- 
(0.8125, 0.1875) -- (.75, 0.1875) -- 
   cycle;
\fill[black]
 (0.9375, 0.1875) -- (0.9375, .25) --
(0.9375, .25) -- (0.9375, .25) --
(1, .25) -- (1, 0.1875) -- 
(1, 0.1875) -- (0.9375, 0.1875) -- 
   cycle;
 \fill[black]
 (0, .75) -- (0, 0.8125) --
(0, 0.8125) -- (.0625, 0.8125) --
(.0625, 0.8125) -- (.0625, .75) -- 
(.0625, .75) -- (0, .75) -- 
   cycle;
\fill[black]
 (0.1875, .75) -- (0.1875, 0.8125) --
(0.1875, 0.8125) -- (.25, 0.8125) --
(.25, 0.8125) -- (.25, .75) -- 
(.25, .75) -- (0.1875, .75) -- 
   cycle;
\fill[black]
 (0, 0.9375) -- (0, 1) --
(0, 1) -- (.0625, 1) --
(.0625, 1) -- (.0625, 0.9375) -- 
(.0625, 0.9375) -- (0, 0.9375) -- 
   cycle;
\fill[black]
 (0.1875, .9375) -- (0.1875, 1) --
(0.1875, 1) -- (.25, 1) --
(.25, 1) -- (.25, .9375) -- 
(.25, .9375) -- (0.1875, .9375) -- 
   cycle;
\fill[black]
 (.75, .75) -- (.75, 0.8125) --
(.75, 0.8125) -- (0.8125, 0.8125) --
(0.8125, 0.8125) -- (0.8125, .75) -- 
(0.8125, .75) -- (.75, .75) -- 
   cycle;
\fill[black]
 (0.9375, .75) -- (0.9375, 0.8125) --
(0.9375, 0.8125) -- (1, 0.8125) --
(1, 0.8125) -- (1, .75) -- 
(1, .75) -- (0.9375, .75) -- 
   cycle;
\fill[black]
 (.75, 0.9375) -- (.75, 1) --
(.75, 1) -- (0.8125, 1) --
(0.8125, 1) -- (0.8125, 0.9375) -- 
(0.8125, 0.9375) -- (.75, 0.9375) -- 
   cycle;
\fill[black]
 (0.9375, .9375) -- (0.9375, 1) --
(0.9375, 1) -- (1, 1) --
(1, 1) -- (1, .9375) -- 
(1, .9375) -- (0.9375, .9375) -- 
   cycle;
\fill[black]
 (3, 0) -- (3, .0625) --
(3, .0625) -- (3.0625, .0625) --
(3.0625, .0625) -- (3.0625, 0) -- 
(3.0625, 0) -- (3, 0) -- 
   cycle;
\fill[black]
 (3.1875, 0) -- (3.1875, .0625) --
(3.1875, .0625) -- (3.25, .0625) --
(3.25, .0625) -- (3.25, 0) -- 
(3.25, 0) -- (3.1875, 0) -- 
   cycle;
\fill[black]
 (3, 0.1875) -- (3, .25) --
(3, .25) -- (3.0625, .25) --
(3.0625, .25) -- (3.0625, 0.1875) -- 
(3.0625, 0.1875) -- (3, 0.1875) -- 
   cycle;
\fill[black]
 (3.1875, 0.1875) -- (3.1875, .25) --
(3.1875, .25) -- (3.25, .25) --
(3.25, .25) -- (3.25, 0.1875) -- 
(3.25, 0.1875) -- (3.1875, 0.1875) -- 
   cycle;
\fill[black]
 (3.75, 0) -- (3.75, .0625) --
(3.75, .0625) -- (3.75, .0625) --
(3.8125, .0625) -- (3.8125, 0) -- 
(3.8125, 0) -- (3.75, 0) -- 
   cycle;
\fill[black]
 (3.9375, 0) -- (3.9375, .0625) --
(3.9375, .0625) -- (3.9375, .0625) --
(4, .0625) -- (4, 0) -- 
(4, 0) -- (3.9375, 0) -- 
   cycle;
\fill[black]
 (3.75, 0.1875) -- (3.75, .25) --
(3.75, .25) -- (3.75, .25) --
(3.8125, .25) -- (3.8125, 0.1875) -- 
(3.8125, 0.1875) -- (3.75, 0.1875) -- 
   cycle;
\fill[black]
 (3.9375, 0.1875) -- (3.9375, .25) --
(3.9375, .25) -- (3.9375, .25) --
(4, .25) -- (4, 0.1875) -- 
(4, 0.1875) -- (3.9375, 0.1875) -- 
   cycle;
 \fill[black]
 (3, .75) -- (3, 0.8125) --
(3, 0.8125) -- (3.0625, 0.8125) --
(3.0625, 0.8125) -- (3.0625, .75) -- 
(3.0625, .75) -- (3, .75) -- 
   cycle;
\fill[black]
 (3.1875, .75) -- (3.1875, 0.8125) --
(3.1875, 0.8125) -- (3.25, 0.8125) --
(3.25, 0.8125) -- (3.25, .75) -- 
(3.25, .75) -- (3.1875, .75) -- 
   cycle;
\fill[black]
 (3, 0.9375) -- (3, 1) --
(3, 1) -- (3.0625, 1) --
(3.0625, 1) -- (3.0625, 0.9375) -- 
(3.0625, 0.9375) -- (3, 0.9375) -- 
   cycle;
\fill[black]
 (3.1875, .9375) -- (3.1875, 1) --
(3.1875, 1) -- (3.25, 1) --
(3.25, 1) -- (3.25, .9375) -- 
(3.25, .9375) -- (3.1875, .9375) -- 
   cycle;
\fill[black]
 (3.75, .75) -- (3.75, 0.8125) --
(3.75, 0.8125) -- (3.8125, 0.8125) --
(3.8125, 0.8125) -- (3.8125, .75) -- 
(3.8125, .75) -- (3.75, .75) -- 
   cycle;
\fill[black]
 (3.9375, .75) -- (3.9375, 0.8125) --
(3.9375, 0.8125) -- (4, 0.8125) --
(4, 0.8125) -- (4, .75) -- 
(4, .75) -- (3.9375, .75) -- 
   cycle;
\fill[black]
 (3.75, 0.9375) -- (3.75, 1) --
(3.75, 1) -- (3.8125, 1) --
(3.8125, 1) -- (3.8125, 0.9375) -- 
(3.8125, 0.9375) -- (3.75, 0.9375) -- 
   cycle;
\fill[black]
 (3.9375, .9375) -- (3.9375, 1) --
(3.9375, 1) -- (4, 1) --
(4, 1) -- (4, .9375) -- 
(4, .9375) -- (3.9375, .9375) -- 
   cycle;
\fill[black]
 (0, 3) -- (0, 3.0625) --
(0, 3.0625) -- (.0625, 3.0625) --
(.0625, 3.0625) -- (.0625, 3) -- 
(.0625, 3) -- (0, 3) -- 
   cycle;
\fill[black]
 (0.1875, 3) -- (0.1875, 3.0625) --
(0.1875, 3.0625) -- (.25, 3.0625) --
(.25, 3.0625) -- (.25, 3) -- 
(.25, 3) -- (0.1875, 3) -- 
   cycle;
\fill[black]
 (0, 3.1875) -- (0, 3.25) --
(0, 3.25) -- (.0625, 3.25) --
(.0625, 3.25) -- (.0625, 3.1875) -- 
(.0625, 3.1875) -- (0, 3.1875) -- 
   cycle;
\fill[black]
 (0.1875, 3.1875) -- (0.1875, 3.25) --
(0.1875, 3.25) -- (.25, 3.25) --
(.25, 3.25) -- (.25, 3.1875) -- 
(.25, 3.1875) -- (0.1875, 3.1875) -- 
   cycle;
\fill[black]
 (.75, 3) -- (.75, 3.0625) --
(.75, 3.0625) -- (.75, 3.0625) --
(0.8125, 3.0625) -- (0.8125, 3) -- 
(0.8125, 3) -- (.75, 3) -- 
   cycle;
\fill[black]
 (0.9375, 3) -- (0.9375, 3.0625) --
(0.9375, 3.0625) -- (0.9375, 3.0625) --
(1, 3.0625) -- (1, 3) -- 
(1, 3) -- (0.9375, 3) -- 
   cycle;
\fill[black]
 (.75, 3.1875) -- (.75, 3.25) --
(.75, 3.25) -- (.75, 3.25) --
(0.8125, 3.25) -- (0.8125, 3.1875) -- 
(0.8125, 3.1875) -- (.75, 3.1875) -- 
   cycle;
\fill[black]
 (0.9375, 3.1875) -- (0.9375, 3.25) --
(0.9375, 3.25) -- (0.9375, 3.25) --
(1, 3.25) -- (1, 3.1875) -- 
(1, 3.1875) -- (0.9375, 3.1875) -- 
   cycle;
 \fill[black]
 (0, 3.75) -- (0, 3.8125) --
(0, 3.8125) -- (.0625, 3.8125) --
(.0625, 3.8125) -- (.0625, 3.75) -- 
(.0625, 3.75) -- (0, 3.75) -- 
   cycle;
\fill[black]
 (0.1875, 3.75) -- (0.1875, 3.8125) --
(0.1875, 3.8125) -- (.25, 3.8125) --
(.25, 3.8125) -- (.25, 3.75) -- 
(.25, 3.75) -- (0.1875, 3.75) -- 
   cycle;
\fill[black]
 (0, 3.9375) -- (0, 4) --
(0, 4) -- (.0625, 4) --
(.0625, 4) -- (.0625, 3.9375) -- 
(.0625, 3.9375) -- (0, 3.9375) -- 
   cycle;
\fill[black]
 (0.1875, 3.9375) -- (0.1875, 4) --
(0.1875, 4) -- (.25, 4) --
(.25, 4) -- (.25, 3.9375) -- 
(.25, 3.9375) -- (0.1875, 3.9375) -- 
   cycle;
\fill[black]
 (.75, 3.75) -- (.75, 3.8125) --
(.75, 3.8125) -- (0.8125, 3.8125) --
(0.8125, 3.8125) -- (0.8125, 3.75) -- 
(0.8125, 3.75) -- (.75, 3.75) -- 
   cycle;
\fill[black]
 (0.9375, 3.75) -- (0.9375, 3.8125) --
(0.9375, 3.8125) -- (1, 3.8125) --
(1, 3.8125) -- (1, 3.75) -- 
(1, 3.75) -- (0.9375, 3.75) -- 
   cycle;
\fill[black]
 (.75, 3.9375) -- (.75, 4) --
(.75, 4) -- (0.8125, 4) --
(0.8125, 4) -- (0.8125, 3.9375) -- 
(0.8125, 3.9375) -- (.75, 3.9375) -- 
   cycle;
\fill[black]
 (0.9375, 3.9375) -- (0.9375, 4) --
(0.9375, 4) -- (1, 4) --
(1, 4) -- (1, 3.9375) -- 
(1, 3.9375) -- (0.9375, 3.9375) -- 
   cycle;
\fill[black]
 (3, 3) -- (3, 3.0625) --
(3, 3.0625) -- (3.0625, 3.0625) --
(3.0625, 3.0625) -- (3.0625, 3) -- 
(3.0625, 3) -- (3, 3) -- 
   cycle;
\fill[black]
 (3.1875, 3) -- (3.1875, 3.0625) --
(3.1875, 3.0625) -- (3.25, 3.0625) --
(3.25, 3.0625) -- (3.25, 3) -- 
(3.25, 3) -- (3.1875, 3) -- 
   cycle;
\fill[black]
 (3, 3.1875) -- (3, 3.25) --
(3, 3.25) -- (3.0625, 3.25) --
(3.0625, 3.25) -- (3.0625, 3.1875) -- 
(3.0625, 3.1875) -- (3, 3.1875) -- 
   cycle;
\fill[black]
 (3.1875, 3.1875) -- (3.1875, 3.25) --
(3.1875, 3.25) -- (3.25, 3.25) --
(3.25, 3.25) -- (3.25, 3.1875) -- 
(3.25, 3.1875) -- (3.1875, 3.1875) -- 
   cycle;
\fill[black]
 (3.75, 3) -- (3.75, 3.0625) --
(3.75, 3.0625) -- (3.75, 3.0625) --
(3.8125, 3.0625) -- (3.8125, 3) -- 
(3.8125, 3) -- (3.75, 3) -- 
   cycle;
\fill[black]
 (3.9375, 3) -- (3.9375, 3.0625) --
(3.9375, 3.0625) -- (3.9375, 3.0625) --
(4, 3.0625) -- (4, 3) -- 
(4, 3) -- (3.9375, 3) -- 
   cycle;
\fill[black]
 (3.75, 3.1875) -- (3.75, 3.25) --
(3.75, 3.25) -- (3.75, 3.25) --
(3.8125, 3.25) -- (3.8125, 3.1875) -- 
(3.8125, 3.1875) -- (3.75, 3.1875) -- 
   cycle;
\fill[black]
 (3.9375, 3.1875) -- (3.9375, 3.25) --
(3.9375, 3.25) -- (3.9375, 3.25) --
(4, 3.25) -- (4, 3.1875) -- 
(4, 3.1875) -- (3.9375, 3.1875) -- 
   cycle;
 \fill[black]
 (3, 3.75) -- (3, 3.8125) --
(3, 3.8125) -- (3.0625, 3.8125) --
(3.0625, 3.8125) -- (3.0625, 3.75) -- 
(3.0625, 3.75) -- (3, 3.75) -- 
   cycle;
\fill[black]
 (3.1875, 3.75) -- (3.1875, 3.8125) --
(3.1875, 3.8125) -- (3.25, 3.8125) --
(3.25, 3.8125) -- (3.25, 3.75) -- 
(3.25, 3.75) -- (3.1875, 3.75) -- 
   cycle;
\fill[black]
 (3, 3.9375) -- (3, 4) --
(3, 4) -- (3.0625, 4) --
(3.0625, 4) -- (3.0625, 3.9375) -- 
(3.0625, 3.9375) -- (3, 3.9375) -- 
   cycle;
\fill[black]
 (3.1875, 3.9375) -- (3.1875, 4) --
(3.1875, 4) -- (3.25, 4) --
(3.25, 4) -- (3.25, 3.9375) -- 
(3.25, 3.9375) -- (3.1875, 3.9375) -- 
   cycle;
\fill[black]
 (3.75, 3.75) -- (3.75, 3.8125) --
(3.75, 3.8125) -- (3.8125, 3.8125) --
(3.8125, 3.8125) -- (3.8125, 3.75) -- 
(3.8125, 3.75) -- (3.75, 3.75) -- 
   cycle;
\fill[black]
 (3.9375, 3.75) -- (3.9375, 3.8125) --
(3.9375, 3.8125) -- (4, 3.8125) --
(4, 3.8125) -- (4, 3.75) -- 
(4, 3.75) -- (3.9375, 3.75) -- 
   cycle;
\fill[black]
 (3.75, 3.9375) -- (3.75, 4) --
(3.75, 4) -- (3.8125, 4) --
(3.8125, 4) -- (3.8125, 3.9375) -- 
(3.8125, 3.9375) -- (3.75, 3.9375) -- 
   cycle;
\fill[black]
 (3.9375, 3.9375) -- (3.9375, 4) --
(3.9375, 4) -- (4, 4) --
(4, 4) -- (4, 3.9375) -- 
(4, 3.9375) -- (3.9375, 3.9375) -- 
   cycle;
\draw (0, 4.5) -- (0.0625, 4.5);
\draw (0.1875, 4.5) -- (0.25, 4.5);
\draw (0.75, 4.5) -- (0.8125, 4.5);
\draw (0.9375, 4.5) -- (1, 4.5);
\draw (3, 4.5) -- (3.0625, 4.5);
\draw (3.1875, 4.5) -- (3.25, 4.5);
\draw (3.75, 4.5) -- (3.8125, 4.5);
\draw (3.9375, 4.5) -- (4, 4.5);
\end{tikzpicture}
\caption{Images of $C_1$, $C_2$, and $C_3$ placed above images of $K_1$, $K_2$, and $K_3$, respectively.}
\label{strips}
\end{figure}

Now we introduce the four-corner Cantor set in the plane.
The first step is to describe the middle-half Cantor set in the real line, denoted by $C$.
For any $n \in \N \cup \set{0}$, let $C_n$ denote the $n^{\text{th}}$ generation of the set $C$.
Then $C_0 = \brac{0,1}$ and for any $n \in \N$, 
$$
C_n=\bigcup_{\substack{a_j\in\{0, 3\}\\j=1, \ldots, n}}\brac{\sum_{j=1}^na_j4^{-j}, \sum_{j=1}^na_j4^{-j}+4^{-n}}.
$$
For example, $C_1 = \brac{0, \frac 1 4} \cup \brac{\frac 3 4, 1}$, the set that is obtained by removing the middle half of $C_0$.
In fact, each $C_{n+1}$ is obtained through the self-similar process of removing the middle half of all intervals that comprise $C_n$.
We define $\disp C  = \bigcap_{n=0}^\iny C_n$, the middle-half Cantor set.
Then the four-corner Cantor set is the product set given by $K=C\times C$. 
Then the $n^{\text{th}}$ generation of $K$ is given by $K_n = C_n \times C_n$, so we may realize the four-corner Cantor set as $\disp K  = \bigcap_{n=0}^\iny K_n$.

The classical Favard length is given by $\disp \Fav\pr{E} = \int_{\mathbb{S}^1} \abs{\proj_\om E} d\om$, where $\proj_\om$ denotes the orthogonal projection onto a line that makes an angle of $\om$ with the $x$-axis, say.
As previously mentioned, it was shown by Nazarov, Peres, and Volberg \cite{NPV10} that $\Fav(K_n)$ exhibits power decay in $n$, see Theorem \ref{NPVThm}.
In \cite{BaV10}, Bateman and Volberg proved a lower bound, described by Theorem \ref{BVThm}, for the rate of decay of the Favard length of the four-corner Cantor set.
These two results are the best known bounds to date.
The point of this article is to provide versions of Theorems \ref{NPVThm} and \ref{BVThm} in which the standard Favard length is replaced by the Favard curve length, as given in Definition \ref{FavC}.

To prove our theorems, we need to impose a number of conditions on the curves that define our projections.
For the upper bound, to ensure that each curve projection is finite, we assume that the curve itself has a finite length.
As the curvature plays an important role in our analysis, this quantity needs to be meaningful.
Therefore, we impose the condition that our curve is piecewise $C^1$ with a piecewise bi-Lipschitz continuous unit tangent vector.
In particular, the unit tangent vector is defined everywhere except for a finite number of points and, by Rademacher's theorem, the curvature is defined a.e. and bounded from above and below.
It follows that the number of points of inflection (points where the signed curvature changes sign) is finite.

The first main result of this article is the following theorem:

\begin{thm}[Upper bound for Favard curve length]
\label{upperBd}
Let $\mathcal{C}$ be a curve of finite length that is piecewise $C^1$ and has a piecewise bi-Lipschitz continuous unit tangent vector.
For every $p>  6$, there exists $C>0$ depending on $\mathcal{C}$ such that for all $n \in \N$,
$$\FavC(K_n) \le C n^{-1/p}.$$
\end{thm}

The second main result of this article is the following theorem:

\begin{thm}[Lower bound for Favard curve length]
\label{lowerBd}
Let $\mathcal{C}$ be a curve that is piecewise $C^1$ and has a piecewise bi-Lipschitz continuous unit tangent vector.
There exists $C>0$ depending on $\mathcal{C}$ such that for all $n \in \N$,
$$\FavC(K_n) \ge C n^{-1}.$$
\end{thm}

If a sharper version of Theorem \ref{NPVThm} became available, our Theorem \ref{upperBd} would automatically inherit this improvement. 
On the other hand, the techniques used to prove Theorem \ref{lowerBd} are more direct, so an improvement to the lower bound for the classical Favard length described by Theorem \ref{BVThm} may or may not affect our lower bound.
Compared to the results of Theorem \ref{BVThm}, our lower bound is weaker.
We are currently investigating whether a $\log n$ improvement may be made to Theorem \ref{lowerBd} as in \cite{BaV10}.  
In a forthcoming paper of Bongers and Taylor \cite{BT21}, an alternate proof of Theorem \ref{lowerBd} is proved using energy techniques.  

To explain why we rule out curves with regions of arbitrarily small curvature, we consider the line segment $\ell = \set{\pr{2t, t} : t \in \pr{-1, 1}}$ and show that $\Fav_{\ell}\pr{K_n}$ does not decay to $0$.
Given any $z \in \brac{0,1}^2$, $z + \ell$ is a line passing through $\brac{0,1}^2$ with slope $\frac 1 2$, see Figure \ref{projPic}.
Since $\proj_{\arctan\pr{1/2}}\pr{K_1}$ fills an interval of length $3/2$, then the line $z + \ell$ intersects $K_1$, so there exists some square $Q_1$ of sidelength $1/4$ that also intersects $z + \ell$.
That is, $z + \ell$ passes through $Q_1$, which is a shifted and rescaled copy of $\brac{0,1}^2$.
By the same reasoning as before, $z + \ell$ also intersects $K_2$ at some some square $Q_2$ of sidelength $1/16$.
Repeating these arguments, we see that $z + \ell$ must also intersect $K_n$.
As this holds for an arbitrary $z \in \brac{0,1}^2$, then we conclude that $\FavC\pr{K_n} \ge 1$, which clearly does not exhibit any decay.

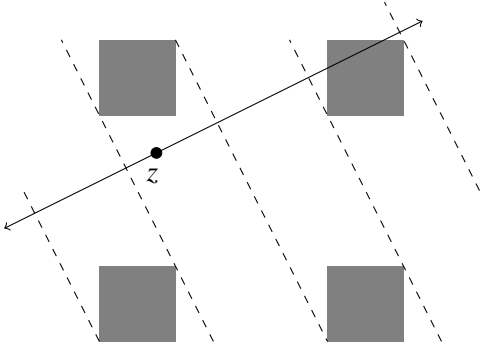
\begin{wrapfigure}{l}{0.55 \textwidth}
\begin{tikzpicture}
\fill[black!50]
 (0, 0) -- (0, 1) --
(0, 1) -- (1, 1) --
(1, 1) -- (1, 0) -- 
(1, 0) -- (0, 0) -- 
   cycle;
\fill[black!50]
 (3, 0) -- (3, 1) --
(3, 1) -- (3, 1) --
(4, 1) -- (4, 0) -- 
(4, 0) -- (3, 0) -- 
   cycle;
\fill[black!50]
 (0, 3) -- (0, 4) --
(0, 4) -- (1, 4) --
(1, 4) -- (1, 3) -- 
(1, 3) -- (0, 3) -- 
   cycle;
\fill[black!50]
 (3, 3) -- (3, 4) --
(3, 4) -- (4, 4) --
(4, 4) -- (4, 3) -- 
(4, 3) -- (3, 3) -- 
   cycle;
\draw [fill=black] (0.75, 2.5) circle (2pt); 
\draw[color=black] (0.7, 2.2) node {$z$};
\draw[->] (0.75, 2.5) -- (4.25, 4.25);
\draw[->] (0.75, 2.5) -- (-1.25, 1.5);
\draw[dashed] (0, 0) -- (-1, 2);
\draw[dashed] (1.5, 0) -- (-0.5, 4);
\draw[dashed] (1, 4) -- (3, 0);
\draw[dashed] (4.5, 0) -- (2.5, 4);
\draw[dashed] (5, 2) -- (3.75, 4.5);
\end{tikzpicture}
\centering
\caption{The image of $z + \ell$ over $K_1$.  
The dotted lines indicate the projection of $K_1$ onto the angle $\arctan\pr{1/2}$.
This shows that even as $z$ moves around $\brac{0,1}^2$, $z + \ell$ must intersect one square from $K_1$.}
\label{projPic}
\end{wrapfigure}

In \cite{DT21}, we prove a quantitative Besicovitch generalized projection theorem using multi-scale analysis. 
As an application, we give an estimate for the rate of decay of the Favard curve length of the four-corner Cantor set.
That is, we show that if $n$ is sufficiently large, then $\FavC\pr{K_n} \lesssim \pr{\log_* n}^{-1/100},$
where $\log_*$ denotes the inverse tower function.
The first result of this paper, Theorem \ref{upperBd}, gives a vast improvement over that estimate from \cite{DT21}.

The article is organized as follows.
In the next section, Section \ref{Curves}, we examine our class of curves.
We make a number of simplifying reductions to streamline our proofs, then we collect some examples of curves that fit into the framework.
Section \ref{UBS} is concerned with the proof of Theorem \ref{upperBd}, while Section \ref{LBS} presents the proof of Theorem \ref{lowerBd}.

\subsection*{Acknowledgements.}
This material is based upon work supported by the National Security Agency under Grant No. H98230-19-1-0119, The Lyda Hill Foundation, The McGovern Foundation, and Microsoft Research, while the authors were in residence at the Mathematical Sciences Research Institute in Berkeley, California, during the summer of 2019.
We would like to express our gratitude to Terrence Tao for reading our manuscript and providing us with helpful feedback.
We would also like to thank the referees for carefully reading our manuscript and offering constructive comments.

\section{The curves}
\label{Curves}

\subsection{Simplifications}

Before proceeding to our proofs, we first make some simplifying assumptions about the class of curves that we consider.

Since $\mathcal{C}$ is assumed to be a curve that is piecewise $C^1$ and has a piecewise bi-Lipschitz continuous unit tangent vector, then we may write the curve $\mathcal{C}$ as a disjoint union of continuous subcurves, $\disp \mathcal{C} = \bigsqcup_{i = 1}^{N} \mathcal{C}_i$, where each $\mathcal{C}_i$ is $C^1$ with a bi-Lipschitz continuous unit tangent vector.
In particular, it can be assumed (by further decomposing the subcurves if necessary) that the unit tangent vectors are strictly monotonic on each $\mathcal{C}_i$.
Since $\mathcal{C}$ is assumed to be of finite length in the upper bound setting, and there is no loss in further assuming that $\mathcal{C}$ is of finite length for the lower bound as well, then each $\mathcal{C}_i$ is of finite length.

We subdivide the unit semi-circle $\brac{0, \pi}$ into 2 parts as follows:
Let $\mathbb{S}^1_x = \brac{0, \pi/4} \cup \brac{3\pi/4, \pi}$ and $\mathbb{S}^1_y = \brac{\pi/4, 3\pi/4}$.
By further decomposing $\mathcal{C}$ if necessary, we may assume the unit tangent vectors of each $\mathcal{C}_i$ are entirely contained in either $\mathbb{S}^1_x$ or $\mathbb{S}^1_y$.

If the unit tangent vectors of $\mathcal{C}_i$ are entirely contained in $\mathbb{S}^1_x$, then we may write $\mathcal{C}_i$ as a graph over $x$.
That is, $\mathcal{C}_i = \set{\pr{t, \vp_i\pr{t}} : t \in I_i}$, where $I_i$ is a finite interval, $\vp_i$ is $C^1$, $\abs{\vp_i'\pr{s}} \le 1$ for all $s \in I_i$, and $\vp_i'$ is $\la_i$ bi-Lipschitz so that for all $s, t \in I_i$,
$$\la_i^{-1}\abs{s-t} \le \abs{\vp_i'\pr{s} - \vp_i'\pr{t}} \le \la_i \abs{s-t}.$$
In particular, $\vp_i'$ is strictly monotonic.
Alternatively, if the unit tangent vectors of $\mathcal{C}_i$ are entirely contained in $\mathbb{S}^1_y$, then we may write $\mathcal{C}_i = \set{\pr{\vp_i\pr{t}, t} : t \in I_i}$, a graph over $y$, with $\vp_i$ and $I_i$ as above.
Since rotating the curve by integer multiples of $\pi/2$ has the same effect as rotating the four-corner Cantor set in the opposite direction by the same multiple of $\pi/2$, such a change does not impact the Favard curve length of $K_n$.
Therefore, there is no loss of generality in assuming that the unit tangent vectors of $\mathcal{C}_1$ are entirely contained in $\mathbb{S}^1_x$.

By subadditivity of the Favard curve length, $\disp \Fav_{\mathcal{C}_1}\pr{E} \le \FavC\pr{E} \le \sum_{i=1}^N \Fav_{\mathcal{C}_i}\pr{E}$.
It follows that, for both the upper and lower bounds that we seek, there is no loss of generality in assuming that $N = 1$.

From now on, we assume that $\mathcal{C} = \set{\pr{t, \vp\pr{t}} : t \in I}$, where $I$ is a finite interval, $\vp$ is $C^1$, $\abs{\vp'} \le 1$ in $I$, $\vp'$ is $\la$ bi-Lipschitz, and $\pr{\vp'}^{-1}$ exists.
In fact, since $\vp'$ is bi-Lipschitz continuous, then $\vp''$ exists a.e., so that $\la \ge \abs{\vp''} \ge \la^{-1} > 0$ a.e. in $I$.

\subsection{Examples}

To conclude this section, we consider some examples of curves that fit into our scheme.

The curve that inspired this work is a circle of radius $R$.
This curve satisfies our hypothesis since it has a finite length equal to $2\pi R$, and is smooth with a smoothly varying tangent vector.
Moreover, the curvature is constant so there are no points of inflection.
While $\mathcal{C} = \set{\pr{R \cos \te, R \sin \te} : \te \in \mathbb{S}^1}$, we may also write $\disp \mathcal{C} = \bigsqcup_{i=1}^4 \mathcal{C}_i$, where
$\mathcal{C}_{2 \pm 1} = \set{\pr{t, \pm R \sqrt{1 - \pr{\frac t R}^2}} : t \in \brac{- \frac R {\sqrt 2}, \frac R {\sqrt 2}}}$ and 
$\mathcal{C}_{3 \pm 1} = \set{\pr{\pm R \sqrt{1 - \pr{\frac t R}^2}, t} : t \in \brac{- \frac R {\sqrt 2}, \frac R {\sqrt 2}}}$.
The functions of interest are $\vp_\pm : \brac{- \frac R {\sqrt 2}, \frac R {\sqrt 2}} \to \R$ defined by $\vp_\pm\pr{t} = \pm \sqrt{R^2 - t^2}$.
Observe that $\vp_\pm$ is smooth over its domain with $\abs{\vp_\pm'} \le 1$ and $\abs{\vp_\pm''} \ge \sqrt 8 /R$.
Therefore, the constants in the estimates for $\FavC\pr{K_n}$ depend only on $R$.

Any ellipse also fits into our scheme.
For some $a, b > 0$, we can write $\disp \mathcal{C} = \bigsqcup_{i=1}^4 \mathcal{C}_i$, where
$\mathcal{C}_{2 \pm 1} = \set{\pr{t, \pm b \sqrt{1 -\pr{\frac t a}^2}} : t \in \brac{- \frac a {\sqrt 2}, \frac a {\sqrt 2}}}$ and 
$\mathcal{C}_{3 \pm 1} = \set{\pr{\pm a \sqrt{1 -\pr{\frac t b }^2}, t} : t \in \brac{- \frac b {\sqrt 2}, \frac b {\sqrt 2}}}$.
It is clear that this curve has a finite length, is appropriately smooth, has no points of inflection, and has a well-defined curvature that is bounded above and below.
By analogy with the previous example, the constants in the estimates for $\FavC\pr{K_n}$ depend only on $a$ and $b$.

We could also consider a logarithmic spiral away from the origin.
That is, suppose that for some $R, k > 0$ and $m > 1$, $\mathcal{C} = \set{\pr{R e^{k \te} \cos \te, R e^{k \te} \sin \te} : \te \in \brac{2\pi, 2m \pi }}$.
This curve is smooth with a finite length and a well-defined curvature that is bounded above and below.
Thus, the estimates for $\FavC\pr{K_n}$ will depend on $n$, $m$ and $R$.

As we discussed after the statement of our theorem, line segments do not always work because there is a special slope at which the Favard lengths associated to such lines do not exhibit any decay.
However, there are many other polynomials that we can work with.
For example, a finite piece of a parabola satisfies the conditions of our theorem.  
Any other higher order polynomial is also suitable, as long as we avoid the points of inflection since those are places at which the curvature smoothly changes sign, and therefore does not have a bi-Lipschitz derivative.

Although it does not satisfy the conditions of our theorem, consider the vertical line segment given by $\ell = \set{\pr{0, t} : t \in \brac{0,1}}$.
If $z \in K_n$, then for any $\be \in \brac{-1, 1}$, $\pr{z_1, \be} + \ell \in K_n$.
It follows that $\Fav_\ell\pr{K_n} \ge 2^{1-n}$, which is a vast improvement over the power decay that we prove in our theorem, even though this curve does not satisfy our hypotheses.

For the curious reader, we mention a curve that does not fit into our scheme is a cycloid.
At the cusps of a cycloid, the curvature blows up, which would affect many of the arguments in the proofs presented below.
However, we could consider the part of the cycloid away from the cusps.

\section{The upper bound}
\label{UBS}

In this section, we prove the upper bound described by Theorem \ref{upperBd}.
The big idea behind the proof is that, locally, we can relate each of the curve projections to an orthogonal projection.
More specifically, we show that for a small enough piece $E$ of $K_n$, given $\al$ in a suitable domain, there exists $\te \in \mathbb{S}^1$ so that the measure of $\Phi_\al\pr{E}$ is comparable to that of $\proj_\te\pr{E}$.
We obtain an explicit relationship between $\te$ and $\al$ and use that $\mathcal{C}$ has uniformly non-vanishing curvature to show that the rate of change of $\al$ with respect to $\te$ is bounded.
This bound allows us to compare an integral of curve projections to that of standard projections, thereby producing a relationship between the Favard curve length of $E$ and the Favard length of $E$. 
By combining these bounds with the result of Nazarov, Peres and Volberg described in Theorem \ref{NPVThm} \cite{NPV10}, we then prove an upper bound for $\FavC\pr{K_n}$.

\begin{wrapfigure}{R}{0.4\textwidth}
\centering{ 
\begin{tikzpicture}
\fill[black, opacity = 0.3]
 (0, 0) -- (0, 1) --
(0, 1) -- (1, 1) --
(1, 1) -- (1, 0) -- 
(1, 0) -- (0, 0) -- 
   cycle;
\draw[color=black] (0.5, 0.5) node {$\widetilde Q_1$};
\fill[black, opacity = 0.3]
 (3, 0) -- (3, 1) --
(3, 1) -- (3, 1) --
(4, 1) -- (4, 0) -- 
(4, 0) -- (3, 0) -- 
   cycle;
\draw[color=black] (3.5, 0.5) node {$\widetilde Q_2$};
\fill[black, opacity = 0.3]
 (0, 3) -- (0, 4) --
(0, 4) -- (1, 4) --
(1, 4) -- (1, 3) -- 
(1, 3) -- (0, 3) -- 
   cycle;
\draw[color=black] (0.5, 3.5) node {$\widetilde Q_3$};
\fill[black, opacity = 0.3]
 (3, 3) -- (3, 4) --
(3, 4) -- (4, 4) --
(4, 4) -- (4, 3) -- 
(4, 3) -- (3, 3) -- 
   cycle;
\draw[color=black] (3.5, 3.5) node {$\widetilde Q_4$};
\fill[black]
 (0, 0) -- (0, .25) --
(0, .25) -- (.25, .25) --
(.25, .25) -- (.25, 0) -- 
(.25, 0) -- (0, 0) -- 
   cycle;
\fill[black]
 (.75, 0) -- (.75, .25) --
(.75, .25) -- (.75, .25) --
(1, .25) -- (1, 0) -- 
(1, 0) -- (.75, 0) -- 
   cycle;
\fill[black]
 (0, .75) -- (0, 1) --
(0, 1) -- (.25, 1) --
(.25, 1) -- (.25, .75) -- 
(.25, .75) -- (0, .75) -- 
   cycle;
\fill[black]
 (.75, .75) -- (.75, 1) --
(.75, 1) -- (1, 1) --
(1, 1) -- (1, .75) -- 
(1, .75) -- (.75, .75) -- 
   cycle;
\fill[black]
 (3, 0) -- (3, .25) --
(3, .25) -- (3.25, .25) --
(3.25, .25) -- (3.25, 0) -- 
(3.25, 0) -- (3, 0) -- 
   cycle;
\fill[black]
 (3.75, 0) -- (3.75, .25) --
(3.75, .25) -- (3.75, .25) --
(4, .25) -- (4, 0) -- 
(4, 0) -- (3.75, 0) -- 
   cycle;
\fill[black]
 (3, .75) -- (3, 1) --
(3, 1) -- (3.25, 1) --
(3.25, 1) -- (3.25, .75) -- 
(3.25, .75) -- (3, .75) -- 
   cycle;
\fill[black]
 (3.75, .75) -- (3.75, 1) --
(3.75, 1) -- (4, 1) --
(4, 1) -- (4, .75) -- 
(4, .75) -- (3.75, .75) -- 
   cycle;
\fill[black]
 (0, 3) -- (0, 3.25) --
(0, 3.25) -- (.25, 3.25) --
(.25, 3.25) -- (.25, 3) -- 
(.25, 3) -- (0, 3) -- 
   cycle;
\fill[black]
 (.75, 3) -- (.75, 3.25) --
(.75, 3.25) -- (.75, 3.25) --
(1, 3.25) -- (1, 3) -- 
(1, 3) -- (.75, 3) -- 
   cycle;
\fill[black]
 (0, 3.75) -- (0, 4) --
(0, 4) -- (.25, 4) --
(.25, 4) -- (.25, 3.75) -- 
(.25, 3.75) -- (0, 3.75) -- 
   cycle;
\fill[black]
 (.75, 3.75) -- (.75, 4) --
(.75, 4) -- (1, 4) --
(1, 4) -- (1, 3.75) -- 
(1, 3.75) -- (.75, 3.75) -- 
   cycle;
\fill[black]
 (3, 3) -- (3, 3.25) --
(3, 3.25) -- (3.25, 3.25) --
(3.25, 3.25) -- (3.25, 3) -- 
(3.25, 3) -- (3, 3) -- 
   cycle;
\fill[black]
 (3.75, 3) -- (3.75, 3.25) --
(3.75, 3.25) -- (3.75, 3.25) --
(4, 3.25) -- (4, 3) -- 
(4, 3) -- (3.75, 3) -- 
   cycle;
\fill[black]
 (3, 3.75) -- (3, 4) --
(3, 4) -- (3.25, 4) --
(3.25, 4) -- (3.25, 3.75) -- 
(3.25, 3.75) -- (3, 3.75) -- 
   cycle;
\fill[black]
 (3.75, 3.75) -- (3.75, 4) --
(3.75, 4) -- (4, 4) --
(4, 4) -- (4, 3.75) -- 
(4, 3.75) -- (3.75, 3.75) -- 
   cycle;
\end{tikzpicture} }
\vspace{2pt}
\caption{Our decomposition of $K_2$ using $K_1$.}
\label{decompPic}
\vspace{-25pt}
\end{wrapfigure}
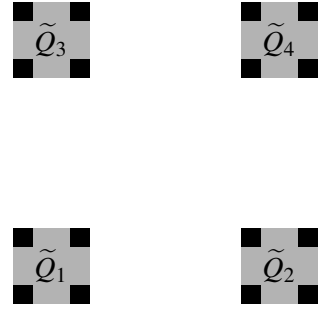

The first step is to decompose $K_n$.
Given the monotonicity of $\FavC\pr{K_n}$, there is no loss is assuming that $n$ is an even number.
Then we rewrite $K_n$ as a collection of rescaled copies of $K_{n/2}$.
To simplify notation, let $\de = 4^{-n}$ so that $\sqrt \de = 4^{-n/2} = 2^{-n}$.
Then
\begin{equation}
\label{Kn2Defn}
K_{n/2} = \bigsqcup_{j=1}^{2^{n}} Q_j, \qquad\qquad\qquad\qquad\qquad\qquad\qquad\qquad
\end{equation}
where $\set{Q_j}_{j=1}^{2^n}$ is a disjoint collection of cubes of sidelength $\sqrt \de$.
For each $j$, define $\widetilde Q_j = K_n \cap Q_j$ so that
\begin{equation*}
K_n = \bigsqcup_{j=1}^{2^{n}} \widetilde Q_j = \bigsqcup_{j=1}^{2^{n}} \pr{K_n \cap Q_j}. \qquad\qquad\qquad\qquad\qquad\qquad\qquad\qquad
\end{equation*}

\begin{example} 
Let $n = 2$ so that $\de = \frac 1 {16}$ and $\sqrt \de = \frac 1 4$.
Then $\disp K_{n/2} = K_1 = \bigsqcup_{j=1}^{4} Q_j$, where each $Q_j$ is a square of sidelength $\frac 1 4$.
Each $\widetilde Q_j = K_2 \cap Q_j$ contains $4$ squares of length $\frac 1 {16}$, so it looks like a scaled, shifted version of $K_1$.
See Figure \ref{decompPic}.
\end{example}

Since each $\widetilde Q_j$ is made up of $2^n$ squares of sidelength $\de$, we may think of each $\widetilde Q_j$ as a shifted, $\sqrt \de$-rescaled copy of $K_{n/2}$.
As the Favard curve length is subadditive (see Definition \ref{FavC}), it follows that 
$\disp \FavC\pr{K_n} \le \sum_{j=1}^{2^n} \FavC\pr{\widetilde Q_j}.$
Therefore, in light of this observation and the simplifying assumptions that we made regarding $\mathcal{C}$ in Section \ref{Curves}, to prove Theorem \ref{upperBd}, it suffices to prove the following proposition:

\begin{prop}[Local Favard curve length]
\label{wQjProp}
Let $\mathcal{C} = \set{\pr{t, \vp\pr{t}} : t \in I}$, where $I$ is a finite interval, $\vp$ is $C^1$, $\abs{\vp'} \le 1$, $\vp'$ is $\la$ bi-Lipschitz, and $\pr{\vp'}^{-1}$ exists.
Decompose $K_n$ as in \eqref{Kn2Defn}.
For any $j \in \set{1, \ldots, 2^n}$ and any $\eps > 0$,
\begin{equation}
\label{FavCtQj}
\FavC\pr{\widetilde Q_j} \lesssim 2^{-n} n^{\eps-1/6},
\end{equation}
where the implicit constant depends on $\mathcal{C}$ and $\eps$.
\end{prop}

One of the main tools used to prove this proposition is a quantitative comparison between each curve projection of $\widetilde Q_j$ and some angular projection of $K_{n/2}$.
The following lemma describes this relationship, which is the important idea behind the whole proof.

\begin{lem}[Comparison between curve projections and orthogonal projections]
\label{projRelLemm}
Let $\mathcal{C} = \set{\pr{t, \vp\pr{t}} : t \in I}$, where $I$ is a finite interval, $\vp$ is $C^1$, $\abs{\vp'} \le 1$, and $\vp'$ is $\la$-Lipschitz, i.e. $\abs{\vp'\pr{s} - \vp'\pr{t}} \le \la \abs{s - t}$ for every $s, t \in I$.
For any $\al \in \R$, any $j \in \set{1, \ldots, 2^n}$, and any $z_0 \in \widetilde Q_j \cap \set{\pr{\al + I} \times \R}$, there exists $\te_{z_0}\pr{\al} \in \mathbb{S}^1$ so that
\begin{align}
\abs{\Phi_\al\pr{\widetilde Q_j}} \simeq 2^{-n} \abs{\proj_{\te_{z_0}\pr{\al}}\pr{K_{n/2}}},
\label{projRelations}
\end{align}
where the implicit constant depends only on $\la$.
\end{lem}

\begin{rem}
\label{two-sided}
The full power of this lemma is not required in our proof.
We only use that $\abs{\Phi_\al\pr{\widetilde Q_j}} \lesssim 2^{-n} \abs{\proj_{\te_{z_0}\pr{\al}}\pr{K_{n/2}}}$ to achieve our result, but we have included the two-sided estimate here anyway.
\end{rem}

\begin{proof}
Fix $j \in \set{1, \ldots, 2^n}$ and $z_0 \in \widetilde Q_j \su \R^2$.
In components, $z_0 = \pr{z_{0,1}, z_{0,2}}$.
Choose $\al \in \R$ so that $z_{0,1} \in \al + I$.
Then $\Phi_\al^{-1}\pr{\Phi_\al\pr{z_0}} = \set{\pr{\al + t, \Phi_\al\pr{z_0} + \vp\pr{t}} : t \in I}$ is the curve passing through $z_0$.
At $z_0$, the slope of the tangent to this curve is given by
$$m_{z_0}\pr{\al} = \vp'\pr{z_{0,1} - \al}.$$

First we describe the set $\Phi_\al\pr{\widetilde Q_j}$ using a $\de$-covering, where $\de = 4^{-n}$.
Since $\mathcal{C} = \set{\pr{t, \vp\pr{t}} : t \in I}$, the graph of a function, then for each point $p = \pr{p_1, p_2}$, the projection is either a singleton or the empty set.
That is,
\begin{equation}
\label{PhialDefn}
\Phi_\al\pr{p} = \left\{ \begin{array}{ll} \set{p_2 - \vp\pr{p_1 - \al}} & p_1 - \al \in I \\ \emptyset & \textrm{ otherwise} \end{array}\right..
\end{equation}
Observe that for any vertical line segment  $v= \set{\pr{q_1, q_2 + t} : 0 \le t \le \de}$ in $\R^2$ of length $\de$, $\Phi_\al\pr{v}$ is either $\emptyset$ or a closed $\de$-interval.  
Since each $\widetilde Q_j$ is a collection of $2^n$ squares of sidelength $\de$, which each contain many such line segments, then $\Phi_\al\pr{\widetilde Q_j} \su \R$ is a finite collection of closed, disjoint intervals, each having length at least $\de$.
Therefore, by a finite version of the Vitali covering lemma, there exists a disjoint collection of $N$ $\de$-intervals, $\disp \set{I_k}_{k=1}^N$, indexed in order, with the property that 
$\disp \bigsqcup_{k=1}^N I_k \su \Phi_\al\pr{\widetilde Q_j} \su \bigcup_{k=1}^N 2 I_k$.
In particular, 
\begin{equation}
\label{projSize}
\abs{\Phi_\al\pr{\widetilde Q_j}} \simeq N \de.
\end{equation}

This $\de$-covering of $\Phi_\al\pr{\widetilde Q_j}$ is now used to understand the set $\widetilde Q_j$ and determine the value of $N$.
We accomplish this by looking at the strips that contain each preimage $\Phi_\al^{-1}\pr{I_k}$.
For each $k \in \set{1, \ldots, N}$, let $T_k = \Phi_\al^{-1}\pr{I_k} \cap Q_j$.
Let $\ell_{z}$ denote the line passing through $z$ with slope $m_{z_{0}}\pr{\al}$.
We define the strip $S_k = \bigcup\set{ \ell_z : z \in T_k}$ to be the smallest strip that runs parallel to the direction $m_{z_{0}}\pr{\al}$ and contains $T_k$.

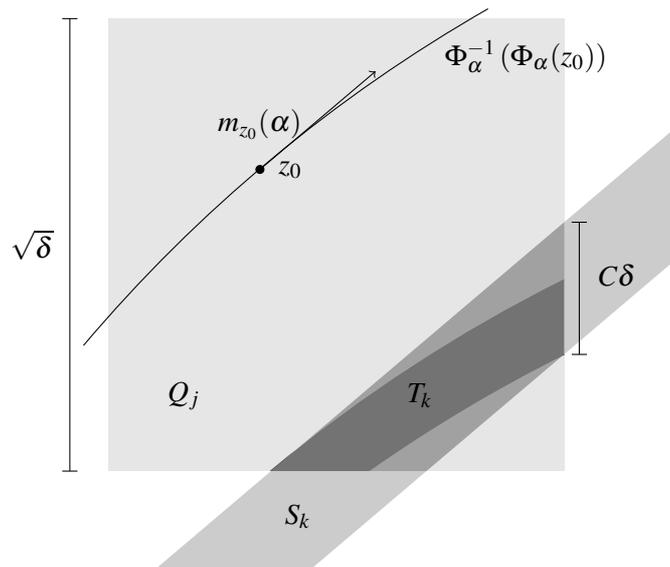
\begin{figure}[h]
\begin{tikzpicture}
\fill[black, opacity = 0.1]
(0, 0) -- (0, 6) --
(0, 6) -- (6, 6) -- 
(6, 6) -- (6, 0) -- 
(6, 0) -- (0, 0) --
cycle;
\draw[color=black] (1, 1) node {$Q_j$};
\draw [fill=black] (2, 4) circle (1.5pt);
\draw[color=black] (2.4, 4) node {$z_0$};
\draw (5, 6.1305) arc (120:140:20);
\draw[color=black] (5.5, 5.5) node {$\Phi_\al^{-1}\pr{\Phi_\al(z_0)}$};
\draw[->] (2,4) -- (3.519, 5.3);
\draw[color=black] (2, 4.6) node {$m_{z_0}(\al)$};
\fill[black, opacity = 0.5]
(3.4428, 0) arc (125.3:116.7:20) --
(6, 1.5474) -- (6,2.5474) --
(6,2.5474) arc (116.7:130:20) --
(2.144, 0) -- (3.4428,0) --
cycle;
\draw[color=black] (4.1, 1) node {$T_k$};
\fill[black, opacity = 0.3]
(2.144, 0) -- (6, 3.3) --
(6, 3.3)-- (6,2.5474) --
(6,2.5474) arc (116.7:130:20) --
cycle;
\fill[black, opacity = 0.3]
(3.4428, 0) arc (125.3:116.7:20) --
(6, 1.5474) -- (4.19 ,0) --
(4.19 ,0) --(3.4428, 0) --
cycle;
\fill[black, opacity = 0.2]
(6, 1.5474) -- (6, 3.3) --
(6, 3.3) --(7.519, 4.6) --
(7.519, 4.6) -- (7.519, 2.8474) --
(7.519, 2.8474) -- (6, 1.5474) --
cycle;
\fill[black, opacity = 0.2]
(2.144, 0) -- (4.19 ,0) --
(4.19 ,0) -- (2.671 ,-1.3) --
(2.671 ,-1.3) -- (0.625, -1.3) --
(0.625, -1.3) -- (2.144, 0) --
cycle;
\draw[color=black] (2.5, -0.6) node {$S_k$};
\draw (-0.5, 0) -- (-0.5, 6);
\draw (-0.6, 0) -- (-0.4, 0);
\draw (-0.6, 6) -- (-0.4, 6);
\draw[color=black] (-1, 3) node {$\sqrt \de$};
\draw (6.2, 1.5474) -- (6.2, 3.3);
\draw (6.1, 1.5474) -- (6.3, 1.5474);
\draw (6.1, 3.3) -- (6.3, 3.3);
\draw[color=black] (6.7, 2.6) node {$C \de$};
\end{tikzpicture}
\caption{The image of one $T_k$ enclosed in $S_k$.}
\label{TkImage}
\end{figure}

We show that each strip $S_k$ has width bounded above by $c\de$.
Recall that $\Phi_\al^{-1}\pr{\be} = \pr{\al, \be} + \mathcal{C}$.
Without loss of generality, $I_k = \brac{0, \de}$ and $Q_j \su \set{0 \le x \le \sqrt \de}$ so that
\begin{align*}
T_k &\su \Phi_\al^{-1}\pr{I_k} \cap \set{0 \le x \le \sqrt \de}
= \set{\pr{\al + t, \be + \vp\pr{t}} : t \in \brac{-\al, \sqrt \de - \al}, \be \in \brac{0,\de}}.
\end{align*}
If $z_i \in T_k$, then $z_i = \pr{\al + t_i, \be_i + \vp\pr{t_i}}$, for some $t_i \in \brac{-\al, \sqrt \de - \al}$ and some $\be_i \in \brac{0,\de}$.
The line $\ell_{z_i}$ is described by
\begin{align*}
y = - \vp'\pr{z_{0,1} - \al} \pr{\al + t_i - x} + \be_i + \vp\pr{t_i}.
\end{align*}
Therefore, given any $z_1, z_2 \in T_k$, the vertical distance between $\ell_{z_1}$ and $\ell_{z_2}$ is given by
\begin{align*}
\dist_y\pr{\ell_{z_1}, \ell_{z_2}}
&= \abs{ - \vp'\pr{z_{0,1} - \al} \pr{\al + t_1} + \be_1 + \vp\pr{t_1} + \vp'\pr{z_{0,1} - \al} \pr{\al + t_2} -  \be_2 - \vp\pr{t_2}} \\
&\le \abs{\vp\pr{t_1} - \vp\pr{t_2} - \vp'\pr{z_{0,1} - \al} \pr{t_1 - t_2} }
+ \abs{ \be_1 - \be_2} \\
&= \abs{\vp'\pr{t_0}\pr{t_1 - t_2} - \vp'\pr{z_{0,1} - \al}\pr{t_1 - t_2}} 
+ \abs{ \be_1 - \be_2} ,
\end{align*}
where we have applied the mean value theorem with $t_0$ as some point between $t_1$ and $t_2$.
The Lipschitz nature of $\vp'$ then implies that
\begin{align*}
\dist_y\pr{\ell_{z_1}, \ell_{z_2}}
&\le  \la \abs{t_0 + \al - z_{0,1}} \abs{t_1 - t_2}
+ \abs{ \be_1 - \be_2} 
\le \pr{\la + 1} \de ,
\end{align*}
where we have used that $t_0 + \al, z_{0,1} \in \brac{0, \sqrt \de}$, $t_1, t_2 \in \brac{-\al, \sqrt \de - \al}$, and $\be_1, \be_2 \in \brac{0, \de}$. 
It follows that the width of $T_k$ (measured orthogonal to $m_{z_0}\pr{\al}$, which is bounded by $1$) is also comparable to $\de$, as desired.

Now we show that the collection $\disp \set{S_k}_{k=1}^N$ is essentially disjoint.
Since $\Phi_\al^{-1}\pr{I_k} \cap \Phi_\al^{-1}\pr{I_{k'}} = \emptyset$ whenever $k \ne k'$ and each $\Phi_\al^{-1}\pr{I_k}$ has a height of $\de$, then $\dist_y\pr{\Phi_\al^{-1}\pr{I_k}, \Phi_\al^{-1}\pr{I_{k'}}} \ge \pr{\abs{k-k'} - 1} \de$.
Since $\abs{m_{z_0}\pr{\al}} \le 1$, then $\dist_{m^\perp}\pr{T_k, T_{k'}} \ge \frac 1 {\sqrt 2} \pr{\abs{k-k'} - 1} \de$, where we use $\dist_{m^\perp}$ to denote the distance measured orthogonal to $m_{z_0}\pr{\al}$.
Thus, whenever $\abs{k - k'} \ge 2\sqrt 2 c+1$, it holds that $\dist_{m^\perp}\pr{T_k, T_{k'}} \ge 2 c \de$.
From the argument in the previous paragraph, for each $k$, $T_k \su S_k$, where $S_k$ is a strip  that is parallel to the slope direction $m_{z_0}\pr{\al}$ and has a width bounded by $c\de$, as depicted in Figure \ref{TkImage}.
It follows that $S_k \cap S_{k'} = \emptyset$ if $\abs{k - k'} \ge 2\sqrt 2 c+1$.
In particular, the strips $\disp \set{S_k}_{k=1}^N$ can have at most $2 \lceil 2\sqrt 2 c+1 \rceil -1$ overlaps, as required.

We may repeat the arguments from above for the dilated intervals.
If we analogously define $T_k^* = \Phi_\al^{-1}\pr{2 I_k} \cap Q_j$, then for each $k$, $T_k^* \su S_k^* \cap Q_j$, where $S_k^*$ is a strip with width bounded by $c^* \de$ that is parallel to $m_{z_0}\pr{\al}$.
Moreover, the collection $\disp \set{S_k^*}_{k=1}^N$ is also essentially disjoint.

Define $\te_{z_0}\pr{\al} \in \mathbb{S}^1$ to be the angle that is orthogonal to a line with slope $m_{z_0}\pr{\al}$. 

\begin{clm}
For each $k$, $\abs{\proj_{\te_{z_0}\pr{\al}}\pr{T_k \cap \widetilde Q_j}} \simeq \de$ and $\abs{\proj_{\te_{z_0}\pr{\al}}\pr{T_k^* \cap \widetilde Q_j}} \simeq \de$.
\end{clm}

Fix $k$.
Since $\pr{T_k \cap \widetilde Q_j}  \su \pr{T^*_k \cap \widetilde Q_j} \su \pr{S^*_k \cap Q_j} \su S^*_k$, then 
$$\abs{\proj_{\te_{z_0}\pr{\al}}\pr{T_k \cap \widetilde Q_j}} \le \abs{\proj_{\te_{z_0}\pr{\al}}\pr{T^*_k \cap \widetilde Q_j}} \le \abs{\proj_{\te_{z_0}\pr{\al}}\pr{S^*_k}}\lesssim \de,$$
where the last inequality follows from the choice of $\te_{z_0}\pr{\al}$ and the fact that $S_k^*$ has width bounded above by $c^*\de$.

Since $I_k \su \Phi_{\al}\pr{\widetilde Q_j}$, then for every $\be \in I_k$, there exists $z \in \widetilde Q_j$ so that $\Phi_{\al}\pr{z} = \be$.
Let $c_k$ denote the midpoint of $I_k$.
As every point of $\widetilde Q_j$ is contained in a $\de$-square, then there exists a $\de$-square $q_i$ such that $q_i \cap \Phi_{\al}^{-1}\pr{c_k} \ne \emptyset$. 
It follows that
$$\abs{\proj_{\te_{z_0}\pr{\al}}\pr{T_k^* \cap \widetilde Q_j}} \ge \abs{\proj_{\te_{z_0}\pr{\al}}\pr{T_k \cap \widetilde Q_j}} \ge \abs{\proj_{\te_{z_0}\pr{\al}}\pr{T_k \cap q_i}} \gtrsim \de,$$
proving the claim.

Finally, we use the claim to conclude the proof.
Recall that, by construction, $\set{2I_k}_{k=1}^N$ forms a cover for $\Phi_{\al}\pr{\widetilde Q_j}$, so that $\disp \widetilde Q_j \su \bigcup_{k=1}^N \pr{T_k^*  \cap \widetilde Q_j}$.
Subadditivity plus an application of the claim shows that
\begin{align*}
\abs{\proj_{\te_{z_0}\pr{\al}}\pr{\widetilde Q_j}}
\le \sum_{k=1}^N \abs{\proj_{\te_{z_0}\pr{\al}}\pr{T_k^*  \cap \widetilde Q_j}}
\lesssim N \de.
\end{align*}
Since $\disp \bigsqcup_{k=1}^N I_k \su \Phi_{\al}\pr{\widetilde Q_j}$ by construction, it follows from taking inverses again that $\disp \bigsqcup_{k=1}^N \pr{T_k  \cap \widetilde Q_j} \su \widetilde Q_j$.
Since each $T_k \su S_k$, where the $S_k$ are essentially disjoint, then another application of the claim shows that
\begin{align*}
\abs{\proj_{\te_{z_0}\pr{\al}}\pr{\widetilde Q_j}}
\gtrsim \sum_{k=1}^N \abs{\proj_{\te_{z_0}\pr{\al}}\pr{T_k  \cap \widetilde Q_j}}
\gtrsim N \de.
\end{align*}
Combining the previous two inequalities shows that $\abs{\proj_{\te_{z_0}\pr{\al}}\pr{\widetilde Q_j}} \simeq N \de$.
However, recalling that $\widetilde Q_j$ is a $\sqrt \de$-scaled, shifted $K_{n/2}$, we have also that $N \de \simeq \abs{\proj_{\te_{z_0}\pr{\al}}\pr{\widetilde Q_j}} = \sqrt \de \abs{\proj_{\te_{z_0}\pr{\al}}\pr{K_{n/2}}}$.
Combining this bound with \eqref{projSize} and recalling the definition of $\de$ leads to the conclusion of the lemma.
\end{proof}

Now that we have Lemma \ref{projRelLemm}, we use it to prove Proposition \ref{wQjProp}.
In essence, we use that $\mathcal{C}$ has non-vanishing curvature to integrate the relationship from Lemma \ref{projRelLemm}.

\begin{proof}[Proof of Proposition \ref{wQjProp}]
The first step is to extend the curve $\mathcal{C}$ to ensure that all curve projections that we are working with are non-empty.
Let $\hat I$ be the $\sqrt \de$-neighborhood of $I$.
That is, if $I = \brac{a, b}$, then $\hat I = \brac{a - \sqrt \de, b + \sqrt \de}$.
Then we extend the definition of $\vp$ from $I$ to $\hat I$ so that all of the properties of $\vp$ are maintained.
That is $\hat \vp : \hat I \to \R$ is $C^1$, $\hat \vp'$ is $\la$ bi-Lipschitz, $\pr{\hat \vp'}^{-1}$ exists, and $\la \ge \abs{\hat \vp''} \ge \la^{-1} > 0$ a.e. in $\hat I$.
For example, we could set $\hat \vp\pr{t} = \left\{\begin{array}{ll} 
\vp\pr{a} + \vp'\pr{a} \pr{t - a} + \frac 1 2 \vp''\pr{a}\pr{t - a}^2 & t \in \bp{a-\de, a} \\
\vp\pr{t} & t \in \brac{a, b} \\ 
\vp\pr{b} + \vp'\pr{b} \pr{t - b} + \frac 1 2 \vp''\pr{b}\pr{t - b}^2 & t \in \pb{b, b+\de} 
\end{array} \right.$, a parabolic extension.
Note that for this choice of extension, $\abs{\hat \vp'\pr{t}} \le 1 + \la \de \le 1 + \la$, which suffices for our arguments.
Each curve projection associated this extended function is denoted by $\hat \Phi_\al$.

Now we proceed with the proof.
Choose $j \in \set{1, \ldots, 2^n}$ and $z_0 \in \widetilde Q_j$.
With $\al \in \R$ so that $z_{0,1}  \in \al + \hat I$, $\te_{z_0}\pr{\al} $ denotes the angle that is orthogonal to a line with slope $m_{z_0}\pr{\al}$.
That is, $\te_{z_0}\pr{\al} \in \pb{-\frac \pi 2, \frac \pi 2}$ is given by
\begin{equation*}
\te_{z_0}\pr{\al} = \arctan\pr{-\frac 1{\hat \vp'\pr{z_{0,1} - \al}}},
\end{equation*}
where we extend the definition of $\arctan$ so that $\arctan\pr{- \frac 1 0} = \frac \pi 2$.
Since $\hat \vp'$ is invertible, if we set 
\begin{align*}
\al_{z_0}\pr{\te}
= z_{0,1} - \pr{\hat \vp'}^{-1}\pr{- \frac{1}{\tan \te}}
= z_{0,1} - \pr{\hat \vp'}^{-1}\pr{- \cot \te},
\end{align*}
we have $\al_{z_0} = \te_{z_0}^{-1}$ and $\te_{z_0} = \al_{z_0}^{-1}$.
That is, $\al_{z_0}\pr{\te_{z_0}\pr{\al}} = \al$ and $\te_{z_0}\pr{\al_{z_0}\pr{\te}} = \te$.
Moreover, since $\hat \vp'$ is a.e. differentiable, then, where it is defined
\begin{align*}
\frac{d\al_{z_0}}{d\te} 
&= - \brac{\hat \vp''\pr{\pr{\hat \vp'}^{-1}\pr{- \cot \te}} \sin^2 \te}^{-1}
=  - \frac{1 + \brac{\hat \vp'\pr{z_{0,1} - \al_{z_0}\pr{\te}}}^2}{\hat \vp''\pr{z_{0,1} - \al_{z_0}\pr{\te}}}. %
\end{align*}
Note that $\set{\al \in \R : \widetilde Q_j \cap \set{\pr{\al + I} \times \R} \ne \emptyset} \su \set{\al \in \R : z_{0} \in \set{\pr{\al + \hat I} \times \R}} =: A_j$.
By set inclusion and the fact that $\Phi_\al\pr{S} \su \hat \Phi_\al\pr{S}$, we see that
\begin{align*}
\FavC\pr{\widetilde Q_j}
&= \int_{\R} \abs{\Phi_\al\pr{\widetilde Q_j}} d\al
= \int_{\set{\al \in \R : \widetilde Q_j \cap \set{\pr{\al + I} \times \R}}} \abs{\Phi_\al\pr{\widetilde Q_j}} d\al
\le \int_{A_j} \abs{\hat \Phi_\al\pr{\widetilde Q_j}} d\al  \\
&\simeq 2^{-n} \int_{A_j} \abs{\proj_{\te_{z_0}\pr{\al}}\pr{K_{n/2}}} d\al,
\end{align*}
where the last line follows from \eqref{projRelations} in Lemma \ref{projRelLemm}.
Since $\al \in A_j$ if and only if $\te_{z_0}\pr{\al} \in T\pr{\de} :=\set{ \arctan\pr{-\frac 1{\hat\vp'\pr{\be}}} : \be \in \hat I}$, then applying a change of variables and the lower bound on $\abs{\vp''}$ shows that
\begin{align*}
\FavC\pr{\widetilde Q_j}
&\lesssim 2^{-n} \int_{T\pr{\de}} \abs{\proj_{\te}\pr{K_{n/2}}} \abs{\brac{\hat \vp''\pr{\pr{\hat \vp'}^{-1}\pr{- \frac{1}{\tan \te}}} \sin^2 \te}^{-1}} d \te \\
&\le 2^{-n} \la \int_{\pb{-\pi/2, \pi/2}} \abs{\proj_{\te}\pr{K_{n/2}}} d \te 
\le  2^{-n} \la \Fav\pr{K_{n/2}}.
\end{align*}
Applying Theorem \ref{NPVThm} leads to \eqref{FavCtQj}, thereby proving the proposition.
\end{proof}

\begin{rem}
If we only had an upper bound for $\abs{\frac{d\al_{z_0}}{d\te}}$, instead of the exact presentation, then the result of Proposition \ref{wQjProp} would still hold.
\end{rem}

\section{The lower bound}
\label{LBS}

In this section, we prove the lower bound that is described by Theorem \ref{lowerBd}.
The starting point of our proof is motivated by the ideas that appear in \cite{BaV10}.
Namely, we introduce a counting function and invoke the Cauchy-Schwarz inequality.
Then the remainder of the proof is concerned with calculating good estimates for the measures of overlapping sets.

We fix $n \in \N$ and proceed to estimate $\FavC(K_n)$ from below. 
As described in Section \ref{Curves}, it suffices to assume that $\mathcal{C} = \set{\pr{t, \vp\pr{t}} : t \in I}$, where $I$ is a finite interval, $\vp$ is $C^1$, $\abs{\vp'} \le 1$ in $I$, and $\vp'$ is $\la$ bi-Lipschitz.
Therefore, for a.e. $s \in I$, $\la^{-1} \le \abs{\vp''\pr{s}} \le \la$ and $\vp''$ does not change sign.
We assume that $\vp'' > 0$ a.e. since the argument for $\vp'' < 0$ is analogous.

Now we introduce the counting function.
Note that $\disp K_n= \bigsqcup_{i = 1}^{4^n} Q_i$, where each $Q_i$ is a cube of sidelength $4^{-n}$. 
For $z \in \R^2$, let $\mathcal{C}_z = z + \mathcal{C}$, the curve positioned at $z$. 
The counting function $f_{n}: \R^2 \to \Z$ is defined by
\begin{equation}
\label{fnDefn}
f_{n}(z)=\#\{\text{cubes }{Q}\in K_n: Q\cap \mathcal{C}_z\ne\emptyset\}.
\end{equation}
We claim that 
$\disp \int_{\R^2} f_{n}(z)\,dz \simeq 1.$
Observe that $\disp f_{n} = \sum_{i = 1}^{4^n} f^i_{n}$, where 
$$f^i_{n}(z)=\left\{ \begin{array}{ll} 1 & \text{ if } Q_i \cap \mathcal{C}_z \ne \emptyset \\ 0 & \text{ otherwise.} \end{array} \right..$$
Since $Q_i \cap \mathcal{C}_z \ne \emptyset$ if and only if $z \in Q_i - \mathcal{C}$, then $f_n^i = \chi_{Q_i - \mathcal{C}}$.
As $\abs{Q_i - \mathcal{C}} \simeq 4^{-n}$ for each $i$, where the implicit constant depends only on $\mathcal{C}$, then the claim follows. 

As in \cite{BaV10}, we apply Cauchy-Schwarz to obtain
\begin{align*}
1 
&\simeq \int_{\R^2}f_n(z)\,dz
\le \abs{\set{z\in \R^2: \mathcal{C}_z\cap K_n\ne\emptyset}}^{1/2} \pr{\int_{\R^2} \abs{f_{n}(z)}^2 dz}^{1/2}
\\
&=\pr{\FavC\pr{K_n}}^{1/2} \pr{\int_{\R^2} \abs{f_{n}(z)}^2 dz}^{1/2}.
\end{align*}
Since $\disp \int_{\R^2} \abs{f_{n}(z)}^2 dz \ne 0$, this gives the lower bound
\begin{equation}
\label{FavLB}
\FavC\pr{K_n} \gtrsim \pr{\int_{\R^2} \abs{f_{n}(z)}^2 dz}^{-1}.
\end{equation}
Therefore, to prove Theorem \ref{lowerBd}, it suffices to estimate $\disp \int_{\R^2} \abs{f_{n}(z)}^2 dz$ from above. 
In particular, we need to establish the following.

\begin{prop}[$L^2$ upper bound for the counting function]
\label{fnL2Prop}
For $f_n$ as defined in \eqref{fnDefn}, it holds that
\begin{equation}
\label{fnL2Bd}
\int_{\R^2} \abs{f_{n}(z)}^2 dz \lesssim n.
\end{equation}
\end{prop}

Recalling the decomposition of $f_n$ from above, we have
\begin{align}
\int_{\R^2} \abs{f_{n}(z)}^2 dz
&= \sum_{i,j = 1}^{4^n}  \int_{\R^2} f^i_{n}\pr{z} f^j_n\pr{z} dz
= \sum_{i,j = 1}^{4^n}  \int_{\R^2} \chi_{Q_i - \mathcal{C}}\pr{z} \chi_{Q_j - \mathcal{C}}\pr{z} dz
= \sum_{i,j = 1}^{4^n} p_{i,j},
\label{fnBound}
\end{align}
where for each pair of cubes $(Q_i, Q_j)$, we have introduced the quantity
$$p_{i,j}=|\pr{Q_i - \mathcal{C}}\cap \pr{Q_j - \mathcal{C}}|.$$

If $i = j$, then it is clear that $p_{i,i}=|Q_i - \mathcal{C}| \simeq 4^{-n}$.
For $k \in \set{0, 1, \ldots, n}$, define the intervals 
$$I_k = \left\{\begin{array}{ll} \brac{\frac 1 {2 \cdot 4^k} + \frac 1 {4^n}, \frac 1 {4^k} - \frac 1 {4^n}} & \text{ if } k < n \\ \set{0} & \text{ if } k = n\end{array}\right. 
\quad \text{ and } \quad 
J_k = \left\{\begin{array}{ll}\brac{\frac 1 {2 \cdot 4^k}, \frac 1 {4^k}} & \text{ if } k < n \\ \brac{0, \frac 1 {4^n}} & \text{ if } k = n \end{array}\right. .$$
For any cube, we can write $Q_i = \pr{x_i, y_i} + \brac{- \frac 1{2 \cdot 4^n}, \frac 1{2 \cdot 4^n}}^2$, where $\pr{x_i, y_i}$ denotes the center of the cube.

\begin{defn}[$\pr{k, \ell}$-pairs]
We say $(Q_i, Q_j)$ is a {\em$\pr{k, \ell}$-pair} for some $k, \ell \in \set{0, 1, \ldots, n}$ if $\abs{x_i - x_j} \in I_k$ and $\abs{y_i - y_j} \in I_\ell$.
It follows that whenever $\pr{\al, \be}\in Q_i$ and $\pr{\ga, \de} \in Q_j$, $\abs{\ga - \al} \in J_k$ and $\abs{\de - \be} \in J_\ell$.
\end{defn}

\begin{example}
For any $i$, the pair $\pr{Q_i, Q_i}$ is an $\pr{n,n}$-pair.
\end{example}

In order to proceed with the proof, we must be able to bound each $p_{i,j}$ from above.
We do this in two steps: first we bound $p_{i,j}$ whenever $\pr{Q_i, Q_j}$ is a $\pr{k, \ell}$-pair, then we count all such pairs.
The following two lemmas give the required quantitative estimates.

\begin{lem}[Measures of overlapping sets]
Let $\pr{Q_i, Q_j}$ be a $\pr{k,\ell}$-pair for some $k, \ell \in \set{0, 1, \ldots, n}$.
If $k \le \ell$, then $p_{i,j} \lesssim 4^{k-2n}$; otherwise, if $k > \ell$, then $p_{i,j} = 0$.
\label{pairBoundsLemma}
\end{lem}

\begin{proof}
Let $\pr{Q_i, Q_j}$ be a $\pr{k,\ell}$-pair.
If $k, \ell = n$, then the result follows from the explanation before the statement, so we assume that $Q_i$ and $Q_j$ are distinct, and then either $k$ or $\ell$ belongs to $\set{0, 1, \ldots, n-1}$.

Observe that $z \in Q - \mathcal{C}$ iff there exists $s \in I$ so that $z + \pr{s, \vp\pr{s}} \in Q$.
Thus, $z \in \pr{Q_i - \mathcal{C}} \cap \pr{Q_j - \mathcal{C}}$ iff there exists $s, t \in I$ so that for some $\pr{\al, \be} \in Q_i$ and $\pr{\ga, \de} \in Q_j$, 
$$\pr{\al - s, \be - \vp\pr{s}} = z = \pr{\ga - t, \de - \vp\pr{t}}.$$
By comparing the coordinates, we see that 
\begin{align*}
&t - s = \ga - \al \\
&\vp\pr{t} - \vp\pr{s} = \de - \be. 
\end{align*}
An application of the mean value theorem shows that for some $\hat s$ between $s$ and $t$,
\begin{align*}
\abs{\vp'\pr{\hat s}}
= \frac{\abs{\vp\pr{t} - \vp\pr{s}}}{\abs{t - s}}
= \frac{\abs{\de - \be}}{\abs{\ga - \al}}. 
\end{align*}
Since $\abs{\vp'\pr{\hat s}} \le 1$ for all $\hat s \in I$, then in order for such a pair of parameters $s$ and $t$ to exist, we must have $\abs{\de - \be} \le \abs{\ga - \al}$.
Since $\pr{Q_i, Q_j}$ is assumed to be a $\pr{k, \ell}$-pair, then $\abs{\ga - \al} \in J_k$ and $\abs{\de - \be} \in J_\ell$, so we see that $k \le \ell$.
Roughly speaking, this means that $p_{i,j}$ is non-zero when the line joining the centers of $Q_i$ and $Q_j$ is closer to being horizontal than vertical.
In particular, if $k > \ell$, then $p_{i,j} = 0$.

Assume that $i$ and $j$ are chosen so that $\pr{Q_i - \mathcal{C}} \cap \pr{Q_j - \mathcal{C}} \ne \emptyset$. 
As shown above, this means that there exist $s_0, t_0 \in I$ and $k \le \ell$ so that 
\begin{align}
&t_0 - s_0 \in \brac{x_j - x_i - \frac 1 {4^n},x_j - x_i + \frac 1 {4^n}} \su J_k \nonumber \\
&\abs{\vp\pr{t_0} - \vp\pr{s_0}} \in \brac{\abs{y_j - y_i} - \frac 1 {4^n}, \abs{y_j - y_i} + \frac 1 {4^n}} \su  J_\ell,
\label{yBds}
\end{align}
where we have assumed, as we may, that $Q_j$ is to the right of $Q_i$. 
Since we have assumed that $Q_i$ and $Q_j$ are distinct, then $k$ belongs to $\set{0, 1, \ldots, n-1}$.

We call $\pr{s, t}$ a {\bf good pair of parameters} if they give rise to a point in the intersection $\pr{Q_i - \mathcal{C}} \cap \pr{Q_j - \mathcal{C}}$.
Let $\mathcal{G} \su I \times I$ denote the set of all good pairs of parameters.
Then $\pr{s_0, t_0} \in \mathcal{G}$.
Now we seek to determine the measure of all $s$ and $t$ for which $\pr{s, t} \in \mathcal{G}$.
We come up with our bound by stepping the pair $\pr{s_0, t_0}$ forward and backward in small steps of length $4^{-n}$.

For $j \in \Z$, let $s_j = s_0 + j 4^{-n}$ and $t_j = t_0 + j 4^{-n}$.
Observe that $t_j - s_j = t_0 - s_0$ for all $j \in \Z$.
\begin{clm}
Whenever $m \in \Z$ is such that $s_{m}, t_{m} \in I$, it holds that $\vp\pr{t_{m}} - \vp\pr{s_{m}} \simeq \vp\pr{t_0} - \vp\pr{s_0} +m 4^{-k-n}$.
\end{clm}

It is clear that this statement holds for $m = 0$.
We first prove by induction that the claim holds for all $m \in \N$ such that $s_{m}, t_{m} \in I$.
The mean value theorem asserts that for some $\hat s_{m} \in \pr{s_{m-1}, s_{m}}$ and some $\hat t_{m} \in \pr{t_{m-1}, t_{m}}$, we have
\begin{align*}
\vp\pr{t_{m}} - \vp\pr{s_{m}}
&= \brac{\vp\pr{t_{m}} - \vp\pr{t_{m-1}}} - \brac{\vp\pr{s_{m}} - \vp\pr{s_{m-1}}} + \brac{\vp\pr{t_{m-1}} -  \vp\pr{s_{m-1}}} \\
&= \vp'\pr{\hat t_{m}}\pr{t_{m} - t_{m-1}} - \vp'\pr{\hat s_{m}}\pr{s_{m} - s_{m-1}} + \brac{\vp\pr{t_{m-1}} - \vp\pr{s_{m-1}}}\\
&\simeq \brac{\vp'\pr{\hat t_{m}} - \vp'\pr{\hat s_{m}}} 4^{-n} + \vp\pr{t_0} - \vp\pr{s_0} + \pr{m-1} 4^{-k-n},
\end{align*}
where we have applied the inductive hypothesis in the last step.
Since $\hat t_{m} - \hat s_{m} \in \brac{t_{m-1} - s_{m}, t_{m} - s_{m-1}} \su \brac{t_0 - s_0 - 4^{-n}, t_0 - s_0 + 4^{-n}}$, then $\hat t_{m} - \hat s_{m} \simeq 4^{-k}$ and the bi-Lipschitz condition on $\vp'$ combined with the assumption that $\vp'$ is increasing implies that $\vp'\pr{\hat t_{m}} - \vp'\pr{\hat s_{m}} \simeq 4^{-k}$.
It follows that $\vp\pr{t_{m}} - \vp\pr{s_{m}} \simeq \vp\pr{t_0} - \vp\pr{s_0} + m 4^{-k-n}$, as claimed.

For $m \in -\N$ such that $s_{m}, t_{m} \in I$, there exists $\hat s_{m} \in \pr{s_{m}, s_{m+1}}$ and $\hat t_{m} \in \pr{t_{m}, t_{m+1}}$ so that
\begin{align*}
\vp\pr{t_{m}} - \vp\pr{s_{m}}
&= \brac{\vp\pr{s_{m+1}} - \vp\pr{s_{m}}} - \brac{\vp\pr{t_{m+1}} - \vp\pr{t_{m}}} + \brac{\vp\pr{t_{m+1}} - \vp\pr{s_{m+1}} } \\
&= \vp'\pr{\hat s_{m}}\pr{s_{m+1} - s_{m}} - \vp'\pr{\hat t_{m}}\pr{t_{m+1} - t_{m}} + \brac{\vp\pr{t_{m+1}} - \vp\pr{s_{m+1}} } \\
&\simeq - \brac{\vp'\pr{\hat t_{m}} - \vp'\pr{\hat s_{m}}} 4^{-n} + \vp\pr{t_0} - \vp\pr{s_0} + \pr{m+1} 4^{-k-n},
\end{align*}
where we have invoked the inductive hypothesis.
Arguing as before, we see that the claim follows for $m \in - \N$ as well.

If $\pr{s_{m}, t_{m}} \in \mathcal{G}$, then by \eqref{yBds} we must have that $\abs{\abs{\vp\pr{t_{m}} - \vp\pr{s_{m}}} - \abs{y_j - y_i}} \le \frac 1 {4^n}$.
It follows from the claim that there exists $U \lesssim 4^{k}$ and $L \gtrsim - 4^{k}$ (not necessarily integers) so that for all $\ga \in \brac{L, U}$, $\pr{s_0 + {\ga} 4^{-n}, t_0 + {\ga} 4^{-n}} \in \mathcal{G}$.
For $\ga \notin  \brac{L, U}$, we either have that the pair does not give rise to an intersection or at least one of the functions is not defined at the corresponding input.
In other words, $\pr{s, s + \pr{t_0 - s_0}} \in \mathcal{G}$ iff $s- s_0 \in \brac{L 4^{-n}, U 4^{-n}}$.
Note that if $\pr{s, t} \in \mathcal{G}$, then $\abs{\pr{s - t} - \pr{s_0 - t_0}} \simeq 4^{-n}$.
It follows that
\begin{align*}
\abs{\set{s : \pr{s, t} \in \mathcal{G} \text{ for some $t \in I$}}} 
\simeq &\abs{\set{t : \pr{s, t} \in \mathcal{G} \text{ for some $s \in I$}}} 
\simeq 4^{k-n}.
\end{align*}

\begin{figure}[h]
\begin{tikzpicture}
\fill[black, opacity = .3]
(0, 0) -- (0, 1) --
(0, 1) -- (1, 1) -- 
(1, 1) -- (1, 0) -- 
(1, 0) -- (0, 0) --
cycle;
\draw[color=black] (0.5, 0.5) node {$Q_i$};
\fill[black, opacity = .3]
(12, 0) -- (12, 1) --
(12, 1) -- (13, 1) -- 
(13, 1) -- (13, 0) -- 
(13, 0) -- (12, 0) --
cycle;
\draw[color=black] (12.5, 0.5) node {$Q_j$};
\fill[black, opacity = 0.2]
(0, 5) arc (90:60:20) --
(10, 2.32) -- (11, 3.32) --
(11, 3.32) arc (60:90:20) --
(1, 6) -- (0, 5)  --
cycle;
\draw[color=black] (1.5, 5.5) node {$Q_i - \mathcal{C}$};
\fill[black, opacity = 0.2]
(13, 5) arc (90:120:20) --
(3, 2.32) -- (2, 3.32) --
(2, 3.32) arc (120:90:20) --
(12, 6) -- (13,5) --
cycle;
\draw[color=black] (11.5, 5.5) node {$Q_j - \mathcal{C}$};
\draw[color=black] (6.5, 4.5) node {$\pr{Q_i - \mathcal{C}} \cap \pr{Q_j - \mathcal{C}}$};
\draw[color=black] (6.5, 1.5) node {$c4^{-k}$};
\draw[-|] (6.9,1.5) -- (12.5, 1.5);
\draw[-|] (6.1,1.5) -- (0.5, 1.5);
\draw[color=black] (-0.5, 0.5) node {$4^{-n}$};
\draw[-|] (-0.5,0.7) -- (-0.5, 1);
\draw[-|] (-0.5,0.3) -- (-0.5, 0);
\draw[color=black] (6.5, 3.5) node {$C 4^{k-n}$};
\draw[-|] (7,3.5) -- (8.5, 3.5);
\draw[-|] (6,3.5) -- (4.5, 3.5);
\draw[color=black] (4, 4.5) node {$4^{-n}$};
\draw[-|] (4,4.7) -- (4, 5);
\draw[-|] (4,4.3) -- (4, 4);
\end{tikzpicture}
\caption{The image of the intersecting set for a $\pr{k, n}$-pair.}
\label{pijImage}
\end{figure}
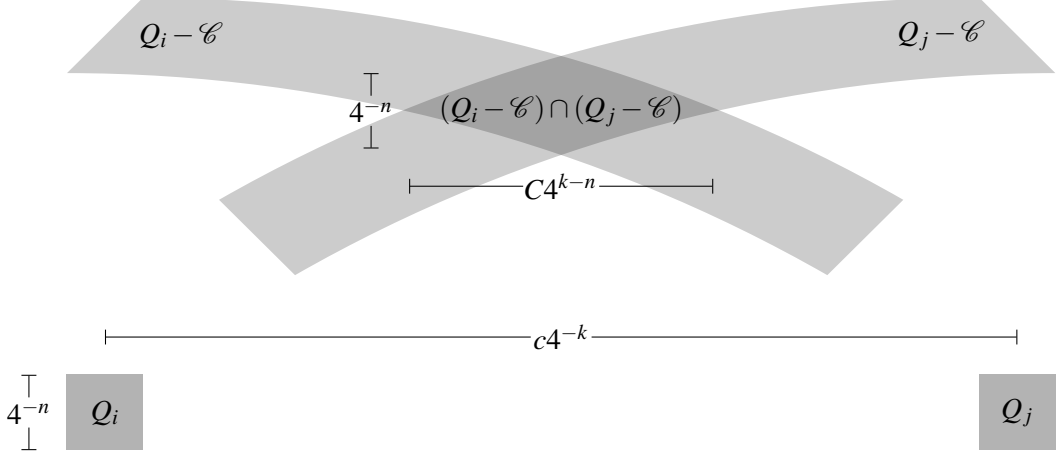

Since the arclength of a piece of $\mathcal{C}$ is proportional to the corresponding parameter range, then we deduce that the width of the intersection of $\pr{Q_i - \mathcal{C}} \cap \pr{Q_j - \mathcal{C}}$ is also bounded above by $C 4^{k-n}$.
Since the height of the intersection of $\pr{Q_i - \mathcal{C}} \cap \pr{Q_j - \mathcal{C}}$ is at most $4^{-n}$, then $\abs{\pr{Q_i - \mathcal{C}} \cap \pr{Q_j - \mathcal{C}}} \lesssim 4^{k-2n}$, as required.
See Figure \ref{pijImage} for a visual in the case where $\ell = n$.
\end{proof}

The next step is to obtain a count on the number of $\pr{k, \ell}$-pairs in $K_n$.
If $\disp C_n = \bigsqcup_{i=1}^{2^n} I_i$, where each $I_i$ denotes an interval of length $4^{-n}$, then we say that $\pr{I_i, I_j}$ is a {\bf $k$-pair} if for any $a \in I_i$ and $b \in I_j$, we have $\abs{a - b} \in J_k$.
It suffices to count the number of $k$-pairs in $C_n$, the $n^{\text{th}}$ generation of the middle-half Cantor set.

\begin{lem}[Pair counting in $C_n$]
For $k \in \set{0, 1, \ldots, n-1}$, $C_n$ contains $2^{2n - 1 -k}$ $k$-pairs, while $C_n$ contains $2^n$ $n$-pairs.
\end{lem}

\begin{proof}
We proceed by induction on $n$.

Since $C_1$ contains $2$ intervals, then $C_1$ contains $4$ pairs.
There are $2$ $1$-pairs (non-distinct pairs) and all of the remaining pairs (of which there are $2$) are $0$-pairs since the distances between their centers equals $3/4$.

Now assume that the statement holds for $C_n$.
Since $C_{n+1}$ contains $2^{n+1}$ intervals, then $C_{n+1}$ contains $2^{n+1}$ $\pr{n+1}$-pairs.
Now consider $k \in \set{1, \ldots, n}$. 
Since $C_{n+1} = \pr{\frac 1 4\cdot C_n} \cup \pr{\frac 3 4 + \frac 1 4\cdot C_n}$, 
then each $k$-pair in $C_{n+1}$ corresponds to a $\pr{k-1}$-pair in one of the two copies of $C_n$.
Thus, by the inductive hypothesis, $C_{n+1}$ contains $2 \cdot 2^{2n - 1 -\pr{k-1}} = 2^{2\pr{n+1} - 1 -k}$ $k$-pairs.
Each of the $0$-pairs comes from choosing one interval in $\frac 1 4\cdot C_n$ and the other in $\frac 3 4 + \frac 1 4\cdot C_n$.
Since $C_n$ contains $2^n$ intervals, there are $2 \cdot 2^n \cdot 2^n = 2^{2n +1}$ $0$-pairs in $C_{n+1}$,
completing the proof.
\end{proof}

Using the count for $k$-pairs in $C_n$, we immediately arrive at the following.

\begin{cor}[Pair counting in $K_n$]
For $k, \ell \in \set{0, 1, \ldots, n-1}$, $K_n$ contains $2^{4n - 2 -k -\ell}$ $\pr{k,\ell}$-pairs.
For $k \in \set{0, 1, \ldots, n-1}$, $K_n$ contains $2^{3n - 1 -k}$ $\pr{k,n}$-pairs.
Further, $K_n$ contains $4^n$ $\pr{n,n}$-pairs.
\label{pairCountingCor}
\end{cor}

Now we have all of the ingredients needed to prove Proposition \ref{fnL2Prop} and therefore complete the proof of Theorem \ref{lowerBd}.

\begin{proof}[Proof of Proposition \ref{fnL2Prop}]
To simplify notation, we write $\pr{Q_i, Q_j} \in \mathcal{P}_{k, \ell}$ if $\pr{Q_i, Q_j}$ is a $\pr{k, \ell}$-pair.
From Lemma \ref{pairBoundsLemma}, recall that $\mathcal{P}_{k, \ell} = \emptyset$ if $k > \ell$.
Returning to equation \eqref{fnBound}, we have
\begin{align*}
\int_{\R^2} \abs{f_{n}(x)}^2 dx
&= \sum_{i,j = 1}^{4^n} p_{i,j}
= \sum_{k=0}^n \sum_{\ell=k}^n \sum_{\pr{Q_i, Q_j} \in \mathcal{P}_{k, \ell}} p_{i,j}
\lesssim  \sum_{k=0}^n \sum_{\ell=k}^n \sum_{\pr{Q_i, Q_j} \in \mathcal{P}_{k, \ell}} 4^{k-2n}, 
\end{align*}
where we have applied Lemma \ref{pairBoundsLemma}.
Continuing on, we see that
\begin{align*}
\int_{\R^2} \abs{f_{n}(x)}^2 dx
&\lesssim \sum_{k=0}^{n-1} \sum_{\ell=k}^{n-1} \sum_{\pr{Q_i, Q_j} \in \mathcal{P}_{k, \ell}} 4^{k-2n} 
+ \sum_{k=0}^{n-1} \sum_{\pr{Q_i, Q_j} \in \mathcal{P}_{k, n}} 4^{k-2n} 
+ \sum_{\pr{Q_i, Q_j} \in \mathcal{P}_{n, n}} 4^{-n} \\
&= \sum_{k=0}^{n-1} \sum_{\ell=k}^{n-1} 2^{4n-2-k-\ell} 4^{k-2n} 
+ \sum_{k=0}^{n-1} 2^{3n-1-k} 4^{k-2n} 
+ 4^n 4^{-n},
\end{align*}
where we have invoked Corollary \ref{pairCountingCor}.
Further simplifying shows that
\begin{align*}
\int_{\R^2} \abs{f_{n}(x)}^2 dx
&\lesssim \sum_{k=0}^{n-1} \sum_{\ell=k}^{n-1} 2^{k-\ell -2} 
+ \sum_{k=0}^{n-1} 2^{k-n-1} 
+ 1 
\lesssim n,
\end{align*}
completing the proof.
\end{proof}

\def\cprime{$'$}


\end{document}